\documentclass[11pt,reqno,a4paper]{amsart}

\usepackage[english]{babel}
\usepackage{graphicx} 

\usepackage{fullpage}
\usepackage{amsmath}
\usepackage{algorithm}
\usepackage{algpseudocode}
\usepackage{amsfonts}
\usepackage[colorlinks=true, allcolors=blue]{hyperref}
\usepackage{xcolor}
\usepackage{comment}
\usepackage{bbm}
\usepackage{tikz}
\usepackage{caption}
\usepackage{subcaption}
\usepackage{bm}
\usepackage{soul}
\usepackage{amsthm}
\usepackage{accents}
\usepackage{mathtools}
\usepackage{amssymb}
\usepackage{booktabs}
\usepackage{float}
\usepackage{csvsimple}
\usepackage{pgfplots}
\usepgfplotslibrary{polar}
\usetikzlibrary{positioning}
\pgfplotsset{compat=1.17}
\usepackage{pgfplotstable}
\usepackage{siunitx}
\usepackage{colortbl}  
\usepackage{makecell}
\usepackage{multirow}
\usepackage{pifont}  
\usepackage{microtype}

\usepackage[capitalize]{cleveref}
\crefname{enumi}{}{}
\crefname{equation}{}{}

\usepackage[style=alphabetic,giveninits=true,maxbibnames=99,maxcitenames=99]{biblatex}
\newtheorem{theorem}{Theorem}[section]
\newtheorem{proposition}[theorem]{Proposition}
\newtheorem{lemma}[theorem]{Lemma}

\newtheorem{assumption}[theorem]{Assumption}
\theoremstyle{definition}
\newtheorem{definition}[theorem]{Definition}
\newtheorem{remark}[theorem]{Remark}

\numberwithin{equation}{section}

\newcommand{\Mid}{\,\middle\vert \,}
\newcommand{\diff}{\, \mathrm{d}}
\newcommand{\widebar}[1]{\overline{#1}}

\renewcommand{\P}{\mathbb{P}}
\DeclareMathOperator{\E}{\mathbb{E}}
\DeclareMathOperator{\R}{\mathbb{R}}

\DeclareMathOperator{\1}{\mathbbm{1}}

\DeclareMathOperator*{\argmax}{arg\,max}

\makeatletter
\def\@tocline#1#2#3#4#5#6#7{\relax
\ifnum #1>\c@tocdepth 
\else
\par \addpenalty\@secpenalty\addvspace{#2}%
\begingroup \hyphenpenalty\@M
\@ifempty{#4}{%
\@tempdima\csname r@tocindent\number#1\endcsname\relax
}{%
\@tempdima#4\relax
}%
\parindent\z@ \leftskip#3\relax \advance\leftskip\@tempdima\relax
\rightskip\@pnumwidth plus4em \parfillskip-\@pnumwidth
#5\leavevmode\hskip-\@tempdima
\ifcase #1
\or\or \hskip 1em \or \hskip 2em \else \hskip 3em \fi%
#6\nobreak\relax
\dotfill\hbox to\@pnumwidth{\@tocpagenum{#7}}\par
\nobreak
\endgroup
\fi}
\makeatother

\title{On the Effectiveness of Classical Regression Methods for Optimal Switching Problems}

\email{\color{blue} martin.andersson@math.uu.se}
\author[Andersson]{Martin Andersson}
\address{Martin Andersson,
Department of Mathematics,
Uppsala University,
S-751 06 Uppsala,
Sweden}
\email{\color{blue} benny.avelin@math.uu.se}
\author[Avelin]{Benny Avelin}
\address{Benny Avelin,
Department of Mathematics,
Uppsala University,
S-751 06 Uppsala,
Sweden}
\email{\color{blue} bakkenolofsson@gmail.com}
\author[Olofsson]{Marcus Olofsson}
\address{Marcus Olofsson,
Department of Mathematics,
Uppsala University,
S-751 06 Uppsala,
Sweden}

\begin{document}

\begin{abstract}
  Simple regression methods provide robust, near-optimal solutions for optimal switching problems, including high-dimensional ones (up to 50). While the theory requires solving intractable PDE systems, the Longstaff-Schwartz algorithm with classical regression methods achieves excellent switching decisions without extensive hyperparameter tuning. Testing linear models (OLS, Ridge, LASSO), tree-based methods (random forests, gradient boosting), $k$-nearest neighbors, and feedforward neural networks on four benchmark problems, we find that several simple methods maintain stable performance across diverse problem characteristics, outperforming the neural networks we tested against. In our comparison, $k$-NN regression performs consistently well, and with minimal hyperparameter tuning. We establish concentration bounds for this regressor and show that PCA enables $k$-NN to scale to high dimensions.
\end{abstract}

\maketitle

\setcounter{tocdepth}{1}
\tableofcontents

\section{Introduction}
Managers of complex production systems face a challenging problem: how to decide when to switch between different operational modes when facing uncertain market conditions and significant switching costs? This question arises across industries, from energy production choosing between different fuel sources to manufacturing plants shifting between different product lines. The mathematical framework of \emph{Optimal Switching} (OS) provides a principled approach to  these decisions, yet a significant gap exists between theoretical solutions and practical implementation. These switching decisions carry substantial economic weight. In energy markets, the difference between optimal and suboptimal switching can represent a considerable loss in annual revenue. A poorly-timed switch can lock a facility into an inefficient mode just as conditions shift, while excessive switching erodes profits through transition costs.

While the mathematical foundations of OS are well-established, the value functions of different operational modes can be represented as solutions to systems of parabolic PDEs with interconnected obstacles, these theoretical solutions remain computationally intractable for realistic problems. These PDE systems are difficult to manage even numerically, and without explicit solutions we cannot construct switching strategies to support actual decision-making. Consider a manager facing just a 10-dimensional state space (comprised of multiple correlated fuel costs, energy prices, weather conditions and forecasts). Traditional finite differences with just 100 grid points in each dimension would require $10^{20}$ points to solve for, exceeding any current computational abilities.

A classical approach to similar problems, particularly in the context of American options, involves simulation and regression methods; see the seminal papers~\cite{LS01} and~\cite{TVR99}. These involve simulating data and then using dynamic programming principles to solve the problem by backward induction, iteratively applying regression methods. These methods shift the computational burden from solving high-dimensional PDEs to a statistical estimation problem: given simulated paths of the underlying stochastic factors, can we instead estimate the value functions? This approach faces its own dimensional challenges but offers more flexibility in addressing them. While this technique has been studied extensively in option pricing, its effectiveness in the more complex OS setting with multiple sequential decisions rather than a single exercise time is underexplored, both theoretically and practically.

This paper studies the Longstaff–Schwartz algorithm, in which the backwards induction of each time-step of the algorithm reduces to a regression problem. In each regression problem, the target is a conditional expectation that ideally would use the true value function. But since this is not available, the algorithm uses an \emph{approximate} value function built from value function-estimates at later time-steps. The theoretical focus of the paper is in quantifying the one-step error between estimates and the approximate value function. This quantity is useful both for assessing performance and for extending to error bounds against the true value function.

We complement the current literature on solving OS problems by simulation and regression, by comparing different regression models; from more classical ones such as linear regression, tree-based models, and $k$-nearest neighbor regression, to more complex feedforward neural networks. We test these models on problems of varying complexity, including high-dimensional settings.

Given the challenging nature of OS problems, especially in higher dimensions, one might expect neural networks to outperform classical models. Somewhat surprisingly, we find that classical models, in particular $k$-nearest neighbor, match or outperform the networks we benchmarked against, while also requiring minimal hyperparameter tuning.

Our contributions in this paper are the following:

\begin{itemize}
  \item \textbf{Theoretical guarantees}: We establish one-step concentration bounds between $k$-NN estimates and the approximate value function, under both diffusion and jump-diffusion dynamics, when the dynamics satisfy either sub-Gaussian or sub-exponential tails (\cref{sec:theoretical}).
  \item \textbf{Comprehensive empirical comparison}: We compare regression models, including linear methods (OLS, Ridge, LASSO), tree-based methods (Random Forests, gradient boosting), $k$-nearest neighbors, and feedforward neural networks, on a suite of management problems from the OS literature, including high-dimensional (up to $d=50$) ones (\cref{sec:numerical}).
  \item \textbf{Simple methods excel}: We find that very simple models, like $k$-NN, can achieve almost optimal management decisions, including in high dimensions, when paired with straightforward dimensionality reduction such as Principal Component Analysis (PCA) (\cref{sec:results}).
\end{itemize}

\section{Literature overview} \label{sec:literature}
The mathematical theory of OS is well studied and contributions have been made by a long list of researchers. Without any aspiration to give a complete list, contributions can be found in, e.g.,~\cite{barkhudaryan2022system,biswas2010viscosity,djehiche2009finite,el2017viscosity,el2009finite,HM13,hu2010multi,K16,lundstrom2014systems,lundstrom2014systems2,lundstrom2022systems,martyr2016dynamic,perninge2018limited,perninge2020finite,perninge2021finite}.
On a more fundamental level, the concepts of backward stochastic differential equations and viscosity solutions for PDEs are of critical importance for the development of OS theory, and we refer to~\cite{CIL92,PP90,PP92} for introductions to these concepts.  As the standard setting has been understood, many alternative formulations of the OS have been considered, e.g., game settings~\cite{HMM19}, delayed decisions and memory~\cite{perninge2018limited,perninge2020finite}, partial information~\cite{LNO15,O18}, state constraints~\cite{K16}, and numerical methods based on simulations~\cite{ACLP12,LOO21}.

When it comes to applications, the literature is far less developed. This contrasts with what could be expected given the importance of closely related topics in financial mathematics; see, e.g.,~\cite{EPQ97}, and~\cite{PS06} for an overview. Typically, the applications considered are low-dimensional and rather unrealistic from a practitioner point of view; see, e.g.,~\cite{BO94,BS85}. A literature survey indicates that, besides~\cite{ACLP12,carmona2008pricing,CL10,LOO21}, concrete applications of optimal switching in at least semi-realistic settings are hard to come by.

A major challenge in the quest for explicit value functions of OS problems is the quickly increasing dimensionality of the problem when applied to more realistic settings. This dimensionality becomes an issue as most numerical methods suffer heavily from the ``curse of dimensionality''. Starting with~\cite{T93}, a classical approach to overcoming this difficulty in option pricing is to ignore the connection to PDEs and instead exploit the stochastic formulation of the optimization problem as an expected value of some functional of a stochastic process; the value function can then be found using Monte Carlo methods. Several attempts have been made to improve this idea using regression, some of the most well-known attempts being~\cite{BG97, C96, HK04, LS01, R02, TVR99}. There have been a few efforts to apply more general regression methods to optimal switching problems; see~\cite{carmona2008pricing,ACLP12,bayraktar2023neural,ludkovski2020simulation}, but the literature is scarce.

The very related topic of optimal stopping has received much more attention, with at least hundreds of variants of Monte Carlo based regression methods, see surveys of \cite{kohler2010review,nadarajah2017comparison}. Attempts were also made already a decade ago to study the optimal stopping problem using neural networks (NN) with a single hidden layer, see~\cite{K10}. Recently, much attention has been directed to this problem from researchers in various fields, including mathematics, finance, data science, and physics. Examples of recent work using deep learning in particular, include~\cite{BCJ19, BCJ21, BH21, EM21, FD21, H20, HKR21,reppen2025neural}.

More generally, deep learning has been applied to stochastic control~\cite{han2016deep} and to solving high dimensional PDEs~\cite{HJW18,sirignano2018dgm}. In the approach of \cite{HJW18}, the equations are solved by reformulating them as backward stochastic differential equations. The work of \cite{hure2020deep} extends the class of equations to ones involving variational inequalities used in optimal stopping, and introduces learning approaches using backwards dynamic programming. The work of \cite{gnoatto2022deep} extends \cite{HJW18} to systems with jumps. These ideas were used in \cite{bayraktar2023neural} to solve OS problems using neural networks, applying the RDBP-algorithm of \cite{hure2020deep} to jump diffusion equations in the OS context. 

Some earlier work on regression-based Monte Carlo methods for optimal switching has focused on a particular regression method, such as global or local polynomial bases in \cite{carmona2008pricing} and \cite{ACLP12} respectively, and only in relatively low dimension ($d=3$ and $d = 6$). In \cite{bayraktar2023neural}, a deep learning approach was investigated in high dimension (up to $d=70$), but comparison with simpler regression models was limited to polynomial bases in specific examples. The work in \cite{ludkovski2020simulation} considered a setup where controls affect part of the state trajectory (unlike our setting) and compared several non-parametric regression approaches, although not neural networks, and only in low dimension. Our work bridges this gap, systematically comparing regression models of varying complexity, including neural networks, in moderate-to-high dimension (up to $d = 50$), while remaining in the classical regression-based Monte Carlo framework of \cite{carmona2008pricing} rather than the particular RDBP-algorithm of \cite{hure2020deep,bayraktar2023neural}, which also involves learning the gradient as well as jump dynamics in the latter case.

Several classical machine learning approaches also offer potential advantages for high-dimensional OS problems. Tree-based methods like random forests~\cite{breiman2001random} and gradient boosting~\cite{friedman2001greedy,ke2017lightgbm} perform implicit feature selection, while regularized linear methods like LASSO~\cite{tibshirani1996regression} and Ridge regression~\cite{hoerl1970ridge} stabilize estimation in high dimensions. Even simple non-parametric methods like $k$-nearest neighbors~\cite{cover1967nearest}, when combined with appropriate dimensionality reduction techniques, can capture local structure in the state space without imposing restrictive functional form assumptions. These methods offer varying trade-offs between interpretability, computational efficiency, and flexibility in high dimensions that may prove advantageous for OS problems compared to more complex deep learning approaches.


Empirical findings across machine learning suggest that simpler methods often match or exceed neural network performance: random forests in classification~\cite{fernandez2014we}, gradient boosting in time series forecasting~\cite{elsayed2021we} and high-dimensional regression~\cite{sheridan2016extreme}, and nearest-neighbor methods in recommendation systems~\cite{ferrari2019we}. Whether these findings extend to optimal switching has not been explored.

\section{Setup of the problem} \label{sec:setup}
\subsection{The general challenge of optimal switching}\label{sec:challenge}
Optimal switching problems with realistic multidimensional state processes resist analytical solution. When the underlying process $X_t$ involves multiple correlated market variables, the curse of dimensionality makes standard finite difference PDE methods computationally infeasible.

A more viable solution is to use Monte Carlo regression methods like Longstaff-Schwartz~\cite{LS01}, but these introduce three major challenges in this setting. First, each regression step produces distorted training targets. Unlike standard supervised learning where true training labels are observed, optimal switching requires estimating value functions that depend on future value functions that are themselves approximated. This nested structure propagates and potentially amplifies approximation errors through the backward induction.

Second, the method requires solving up to $N \times D$, where $N$ is number of time-steps in discretization of the problem and $D$ the number of production modes, separate regression problems. With typical problems involving hundreds of time steps and multiple modes, computational scaling becomes a dominant constraint on method selection.

Third, each individual regression model must be able to handle high-dimensional feature spaces. Many applications can involve 10--50 market variables (electricity prices across regions, multiple fuel costs, weather patterns, demand forecasts), creating feature spaces where many methods suffer from the curse of dimensionality or require prohibitively large training datasets.

\begin{remark}[On the number of regressions]
  Some regression models can output vectors, and this would reduce the number of regression problems to $N$. However, this reduction might also affect performance. Let us also note that as long as decisions do not need to be made too frequently, perhaps just daily, regressing $N \times D$ times does not necessarily become prohibitive. Further, if one needs a finer discretization for simulation purposes, we can adapt the learning algorithm to only allow for decisions on a subset of $\{1,\dots,N\}$. 
\end{remark}

\subsection{Production facility setup}
The general OS problem we consider in this paper involves a production facility operating in $D$ modes under random market conditions $X_t \in \mathbb{R}^d$. Let $X_t$ be a continuous-time Markov process representing market conditions (electricity prices, fuel costs, weather). Each production mode $i = 1,\ldots,D$ generates instantaneous payoff $f_i(t,x)$ at time $t$ in state $x$, with switching costs $c_{ij}(t,x)$ between modes.

A production strategy is a random function $\mu(t)$ indicating the active production mode at time $t$, adapted to the process $X_t$. This corresponds to a sequence of (stopping) switching times $\{\tau_k\}$ and modes $\{\xi_k\}$ (post-switch modes):
\begin{equation}
  \mu(t) = \xi_k \quad \text{for} \quad \tau_k \leq t < \tau_{k+1}.
\end{equation}

Further, to prevent multiple instantaneous switches being part of an optimal strategy, these assumptions on the cost functions $c_{ij}$ are standard:
\begin{assumption}\label{ass:costs}
  For all pairwise distinct modes $i, j, k$ and all $(t,x) \in [0,T] \times \mathbb{R}^d$:
  \begin{enumerate}
    \item $c_{ii}(t,x) = 0$ (no cost for staying in the same mode),
    \item $c_{ij}(t,x) > 0$ for $i \neq j$ (positive switching costs),
    \item $c_{ij}(t,x) + c_{jk}(t,x) \geq c_{ik}(t,x)$ (triangle inequality).
  \end{enumerate}
\end{assumption}
\subsection{The value function}
The total expected payoff for a given strategy $\mu$ starting in mode $i$ at time $t$ and state $x$ is:
\begin{equation*}
  J_i(t,x,\mu) = \E \left[ \int_t^T f_{\mu(s)}(s,X_s) \diff s - \sum_{T \geq \tau_k > t} c_{\xi_{k-1},\xi_k}(\tau_k,X_{\tau_k}) \Mid X_t = x\right].
\end{equation*}

The value function is the supremum over all admissible strategies:
\begin{equation}\label{eq:OSPvalue}
  V_i(t,x) = \sup_{\mu} J_i(t,x,\mu).
\end{equation}

It is a classical result of OS that, under rather weak conditions (see for instance~\cite{tang1993finite}), this value function is the viscosity solution of the following system of PDEs
\begin{equation} \label{eq:formOS}
  \begin{split}
     & \min \{-\partial_t V_i - \mathcal{L} V_i - f_i, V_i - \max_{j \neq i} \{V_j - c_{ij}\}\} = 0, \\
     & V_i(T,x) = 0,
  \end{split}
\end{equation}
where $\mathcal{L}$ is the infinitesimal generator of the process $X_t$.

Once one has access to the value functions $V_i(t,x)$, the optimal strategy is characterized as follows: when in mode $i$, switch to mode $\hat{\jmath}$ as soon as the process $V_i(t, X_t)$ hits its obstacle $\max_{j=1,\dots,D;j\neq i}\{V_j - c_{ij}\}$, i.e., at the first time $\tau$ such that
\begin{equation}\label{eq:optimal}
  V_i(\tau, X_{\tau}) \leq \max_{j=1,\dots,D;j\neq i} \{V_j(\tau,X_{\tau})- c_{ij}(\tau,X_{\tau})\},
\end{equation}
and $\hat{\jmath} \in \arg \max_{j=1,\dots,D;j\neq i} \{V_j(\tau,X_\tau)- c_{ij}(\tau,X_\tau)\}$.
\subsection{Numerical solution approach} \label{sec:setup:numerical}
In this section we detail the Longstaff-Schwartz approach we use in this paper.

\subsubsection{Process and discretization}
The underlying market process follows jump-diffusion dynamics, where jumps capture sudden market disruptions from weather events, plant outages, or demand spikes:
\begin{equation} \label{eq:process}
  dX_t = b(t,X_t) \diff t + \sigma(t,X_t) \diff W_t + \gamma(t,X_t) \diff J_t,
\end{equation}
where $b(t,x)$ is a drift term, $\sigma(t,x)$ is a diffusion term, $\gamma$ are the jump sizes, $J_t$ is a jump process and $W_t$ is $d$-dimensional Brownian motion. In our experiments $\gamma$ will be a diagonal matrix and each component of $J_t$ will have the form
\begin{equation} \label{eq:jump}
  J_t^i = \sum_{j=1}^{\infty} Y_j^i\1_{\{T_j^i \leq t\}},
\end{equation}
where $Y_j^i$ are the jump sizes and $T_j^i$ is the time of the $j$-th jump of the $i$-th component of the process. The jump process $J_t$ is assumed to be independent of the Brownian motion and the jump times are assumed to be independent of each other.

We discretize the time $[0,T]$ under consideration into $\mathcal T= \{0= t_0, t_1, \dots, t_N = T\}$ with $\Delta t = T/N$ and only allow mode switching at these times, reflecting operational constraints in real production facilities. The continuous parts of the paths $X_{t_n}$, $n=0,\dots,N$, are simulated according to the standard Euler-Maruyama scheme, and at each time step, we simulate the number of jumps and add their cumulative impact to the diffusion update.

Now let ${\mu}$ be a (discrete) management strategy with switching times $\tau_0,\tau_1,\dots,\tau_{N-1}$. The (discretized) expected profit using $\mu$ from $t=t_n$ to $T=t_N$ is then (denoting it again with $J_i$)
\begin{equation*}
  {J}_i(t_n,x,{\mu}) = \E \left[ \sum_{k=n}^{N-1}\Delta t f_{{\mu}_k}(t_k,X_{t_k}) - \sum_{\tau_k\geq t} c_{\xi_{k-1},\xi_{k}}(\tau_k,X_{\tau_k}) \Mid X_{t_n} = x,\xi_{0} = i \right],
\end{equation*}
and the (discretized) value function is defined by (denoting it again with $V_i$)
\begin{equation}\label{value:rec}
  V_i(t_n,x) =  \max_{\widetilde{\mu}}{J}_i(t_n,x,{\widetilde \mu}).
\end{equation}

\subsubsection{The Regression Problem}\label{sec:regression}
The value function satisfies a dynamic programming relation:
\begin{equation} \label{eq:dpp}
  {V}_i(t_n,x) = \max_{j=1,\dots,D} \left[ \Delta t \, f_j(t_n,x)
    - c_{ij}(t_n,x) + \E\left[{V}_j(t_{n+1},X_{t_{n+1}})
    \,\middle\vert\, X_{t_n} = x\right] \right].
\end{equation}
The Longstaff-Schwartz algorithm estimates each continuation value $\E[V_j(t_{n+1},X_{t_{n+1}}) \mid X_{t_n}=x]$ separately via regression. The transition kernel admits a density in our setting, so that 
\begin{equation*}
\E\left[{V}_j(t_{n+1},X_{t_{n+1}})\,\middle\vert\, X_{t_n} = x\right] = \int_{\R^d}{V}_j(t_{n+1},y)\rho(x,t_n;y,t_n+\Delta t)\,dy,
\end{equation*}
where $\Delta t=t_{n+1}-t_n$. A fundamental challenge here is that future value functions $V_j(t_{n+1}, \cdot)$ are unknown, and we can only use noisy and possibly biased approximations $\widehat{V}_j(t_{n+1}, \cdot)$ from previous time steps. Thus, the target we estimate at each backwards step is
\begin{equation} \label{eq:gj}
  g_j(t_n,x) = \int_{\R^d}{\widehat V}_j(t_{n+1},y)\rho(x,t_n;y,t_n+\Delta t)\,dy,
\end{equation}
a random function which depends on previous estimates.
Since we do not have direct access to $g_j$, we simulate the transition and train regression models $\mathcal{R}_{n,j}$ that map state $x$ to an estimate of $g_j(t_n,x)$. For convenience, we write $\mathcal{R}_j(t_n,x) = \mathcal{R}_{n,j}(x)$ and denote the collection of all trained models by $\mathcal{R}$. Then
\begin{equation}\label{eq:Vhat:reconstruction}
  \widehat{V}_i(t_n,x) = \max_{j=1,\dots,D}\left[\Delta t\,f_j(t_n,x)
    - c_{ij}(t_n,x) + \mathcal{R}_j(t_n,x)\right].
\end{equation}

To further set up our approach, let $\mathcal{X} = \{X^{s}_{t_0},\ldots,X^{s}_{t_N}; s=1,\ldots,M\}$ be the set of all states at all times for $M \in \mathbb{N}$ independent trajectories defined in the previous subsection. The processes' starting points are drawn from a distribution $\mu_0$ which is assumed to be sub-Gaussian, see \cref{def:sub-gaussian}. Then the estimate $\mathcal{R}_j(t_n,x)$ of $g_j(t_n,x)$ is obtained by regression over $M$ sample pairs:
\begin{equation}\label{eq:regressionPairs}
  (X_{t_n}^s, Y_{j,t_n}^s), \quad Y_{j,t_n}^s = \frac{1}{m_Y} \sum_{\hat{s}=1}^{m_Y}
    \widehat{V}_j(t_{n+1}, \widetilde{X}^{s,\hat{s}}_{t_{n+1}}),
\end{equation}
where $\widetilde{X}^{s,\hat s}_{t_{n+1}}$, $\hat{s} = 1,\dots,m_Y$ are one-step transitions starting in $X_{t_n}^s$. Conditional on $\mathcal{X}$, the transitions ${\widetilde X^{s,\hat s}_{t_{n+1}}}$ are mutually independent across $(n,s,\hat s)$, with $\widetilde X^{s,\hat s}_{t_{n+1}}$ distributed as one step of the transition kernel from $X^s_{t_n}$. Here $m_Y$ trades computational cost against Monte Carlo variance, typically $m_Y=1$ in our simulations, but larger values are implied in our concentration bounds in \cref{sec:theoretical}.
\subsection{Main algorithm and computational complexity}\label{sec:algorithm}
Computational scaling creates a fundamental trade-off between approximation quality and practical feasibility. The algorithm requires training $N \times D$ separate regression models, one for each time step and operating mode. With hundreds of time steps and multiple modes, this massive scaling means that a method's computational efficiency must be taken into account along with theoretical approximation properties in determining overall performance. A simple method that trains quickly across hundreds of models can be preferable to a more sophisticated method that provides better individual approximations but becomes computationally prohibitive when scaled.

Now follows the complete training algorithm:
\begin{algorithm}[H]
  \caption{Value Function Approximation via Regression}
  \label{alg:main}
  \begin{algorithmic}[1]
    \State Generate $M$ simulated state trajectories $\{X_{t_0}^s, \ldots, X_{t_N}^s\}_{s=1}^M$
    \State Set terminal condition: $\widehat{V}_{j}(t_N,x) \gets 0$ for all $j=1,\ldots,D$
    \For{$n = N-1, N-2, \ldots, 0$} \Comment{Backward induction}
    \State Extract states $\{X_{t_n}^1, \ldots, X_{t_n}^M\}$ at time $t_n$
    \State For each $X_{t_n}^s$, simulate $m_Y$ one-step transitions $\{\widetilde{X}_{t_{n+1}}^{s,\hat{s}}\}_{\hat{s}=1}^{m_Y}$
    \For{$j = 1, \ldots, D$} \Comment{For each target mode}
    \State Compute targets: $Y_{j,t_n}^s = \frac{1}{m_Y}\sum_{\hat{s}=1}^{m_Y} \widehat{V}_j(t_{n+1},\widetilde{X}_{t_{n+1}}^{s,\hat{s}})$
    \State Train regression model $\mathcal{R}_{n,j}$ using $\{(X_{t_n}^s, Y_{j,t_n}^s)\}_{s=1}^M$
    \EndFor
    \For{$j = 1, \ldots, D$} \Comment{Reconstruct value functions}
    \State $\widehat{V}_j(t_n,x) \gets \max_{k=1,\dots,D}[\Delta t\, f_k(t_n,x) - c_{jk}(t_n,x) + \mathcal{R}_{k}(t_n,x)]$
    \EndFor
    \EndFor
  \end{algorithmic}
\end{algorithm}
The total computational cost is $O(N \times D \times (C_{\text{train}}(M,d) + M \times m_Y \times C_{\text{pred}}(d)))$, where $C_{\text{train}}$ and $C_{\text{pred}}$ are the training and prediction complexities for each model. Training complexity typically dominates, but the $M \times m_Y$ prediction calls during backward induction can create bottlenecks for expensive inference methods like $k$-nearest neighbors. The specific complexities for each method appear in \cref{tab:complexity_comparison}.

\subsection{Metrics and comparisons}\label{sec:metrics}
Since the true value function $V_i(t,x)$ is unknown, we evaluate our models using upper and lower bounds constructed from benchmark strategies.

\textbf{Upper bound: A posteriori strategy.} The \emph{a posteriori strategy} $\mu^{\text{ap}}$ uses future knowledge to make optimal decisions, clearly inadmissible since it is not adapted to the process, but provides an upper bound on achievable performance. For a realized trajectory $\{X_{t_0},X_{t_1},\ldots,X_{t_N}\}$, we define by backwards induction:
\begin{equation*}
  \mathcal{V}^{\textnormal{ap}}_{i}(t_n) = \max_{j \in \{1, \ldots, D\}} \left\{ f_j(t_n,X_{t_n}) \Delta t - c_{ij}(t_n,X_{t_n}) + \mathcal{V}^{\textnormal{ap}}_{j}(t_{n+1}) \right\},
\end{equation*}
where $\mathcal{V}^{\textnormal{ap}}_{i}(t_N) = 0$ for all $i$. The corresponding strategy $\mu^{\textnormal{ap}}(t_n)$ follows from \cref{eq:optimal}, yielding the a posteriori value function:
\begin{equation}\label{eq:ap:val}
  V^{\text{ap}}_i(t_n, x) = \E \left[ \sum_{k=n}^{N-1}\Delta t f_{\mu^{\textnormal{ap}}_k}(t_k,X_{t_k}) - \sum_{\tau_k\geq t} c_{\xi_{k-1},\xi_{k}}(\tau_k,X_{\tau_k}) \Mid X_{t_n} = x,\xi_{0} = i \right].
\end{equation}
By Jensen's inequality $V_i(t_n,x) \leq V_{i}^{\text{ap}}(t_n,x)$, no admissible strategy can beat complete knowledge of the future.

\textbf{Lower bound: Myopic strategy.} The greedy strategy ignores future consequences, maximizing only immediate payoffs. For $n < N$:
\begin{equation}\label{eq:greedy:strategy}
  \mu^{\text{gr}}(t_n) \in \argmax_{j \in \{1, \ldots, D\}}
  \left\{ f_j(t_n,X_{t_n}) \Delta t - c_{\mu^{\text{gr}}(t_{n-1}), j}(t_n,X_{t_n}) \right\}.
\end{equation}
The greedy value function is:
\begin{equation}\label{eq:greedy:val}
  V^{\text{gr}}_i(t_n,x) = J_i(t_n,x,\mu^{\text{gr}}(t_n)).
\end{equation}
Since this strategy is adapted by construction, $V_i(t,x) \geq V_{i}^{\text{gr}}(t,x)$. We estimate both \cref{eq:ap:val,eq:greedy:val} by their empirical means.

\textbf{Strategy implementation.} We implement strategies $\mu^\mathcal{R}$ derived from our regression models using \cref{eq:optimal}. Let $S_{j} = \widehat{V}_{j}(t_{n},X_{t_n}) - c_{\mu^\mathcal{R}(t_{n-1})j}(t_n,X_{t_n})$, where $\widehat{V}_j$ is reconstructed from the trained continuation values via \cref{eq:Vhat:reconstruction}. Then:
\begin{align}
  \mu^{\mathcal{R}}(t_n) =
  \begin{cases}
    \arg\max_{j \neq \mu^{\mathcal{R}}(t_{n-1})} S_j & \text{if } \max_{j \neq \mu^{\mathcal{R}}(t_{n-1})} S_j \geq S_{\mu^{\mathcal{R}}(t_{n-1})} \\
    \mu^{\mathcal{R}}(t_{n-1})                       & \text{otherwise}
  \end{cases}.
\end{align}
Let us also define $V^{\mathcal{R}}_i(t_n,x) = J_i(t_n,x,\mu^{\mathcal{R}}(t_n))$, the expected value achieved by following the adapted strategy $\mu^{\mathcal{R}}$.

\textbf{Model benchmarking.} Any reasonable regression model $\mathcal{R}$ should satisfy:
\begin{equation}
  V_i^{\text{gr}}(t_n,x) \leq V^{\mathcal{R}}_i(t_n,x) \leq V_i^{\text{ap}}(t_n,x).
\end{equation}
Since $\mu^{\text{gr}}$ is an admissible strategy, achieving the lower bound verifies non-trivial learning. Further, because $V_i^{\mathcal{R}} \leq V_i \leq V_i^{\text{ap}}$, approaching the upper bound indicates near-optimal performance, since this implies $V_i^{\mathcal{R}}$ is necessarily close to the true value function $V_i$.

\textbf{Performance metrics.} We evaluate models on three dimensions:

\emph{1. Decision quality:} How often does our strategy match a posteriori-optimal switches?
\begin{equation} \label{eq:decision_similarity}
  Q_i^\mathcal{R} \coloneq \E \left[  \frac{\sum_{\tau\in \mathcal{T}} \1_{\mu^{\mathcal{R}}(\tau) = \mu^{\text{ap}}(\tau)} }{N+1} \Mid \xi_0 = i\right].
\end{equation}

\emph{2. Value capture:} What fraction of achievable value do we realize? We normalize the realized value:
\begin{equation} \label{eq:normalized_value}
  \kappa_i^\mathcal{R}(t_n,x) \coloneq \frac{V^{\mathcal{R}}_i(t_n,x)}{V^{\text{ap}}_i(t_n,x)}.
\end{equation}
\begin{remark}
  In all the numerical experiments in \cref{sec:numerical}, the a posteriori value $V_i^{\text{ap}}$ and the realized value $V_i^{\mathcal{R}}$ at the evaluated initial states are positive, so the reported value capture $\kappa_i^{\mathcal{R}}$ lies in $[0,1]$ (see \cref{sec:tables}).
\end{remark}
\emph{3. Internal consistency:} Does our value function accurately predict its own strategy's performance?
\begin{equation} \label{eq:prediction_accuracy}
  C_i^\mathcal{R}(t_n,x) \coloneq \frac{1}{1+|V^{\mathcal{R}}_{i}(t_n,x)- \widehat{V}_i(t_n,x)|/|V^{\text{ap}}_i(t_n,x)|}.
\end{equation}
While decision quality and value capture measure actual performance, this measures another useful aspect: model calibration versus discrimination. A model can make good relative decisions (high decision quality) while having systematically biased value estimates. For instance, if a model overestimates all values by 50\%, it may still rank decisions correctly and achieve good value capture, but managers cannot trust the absolute value predictions for budgeting, risk assessment, or comparing with alternative investments.

\section{Regression models for optimal switching}\label{sec:models}
The three challenges established in \cref{sec:challenge}, distorted training targets, $N\times D$ computational scaling, and high-dimensional feature spaces, motivate our selection of regression methods. We evaluate several regression approaches: linear models as natural baselines, local regression methods ($k$-NN), tree-based methods (Random Forests, LGBM), and neural networks, each offering different tradeoffs in robustness, scaling and handling of high-dimensional feature spaces.

\textbf{Computational scaling}
Training $N\times D$ models makes per-model training cost critical. \Cref{tab:complexity_comparison} shows the scaling properties that determine practical feasibility. Methods like polynomial OLS with degree $p$ suffer $O(Md^{2p})$ training costs that explode exponentially with both dimension and polynomial degree, making it completely infeasible beyond low-dimensional problems.

\begin{table}[htbp]
  \centering
  \caption{Computational Complexity by Method}
  \label{tab:complexity_comparison}
  \begin{tabular}{llll}
    \hline
    \textbf{Method}  & \textbf{Training}            & \textbf{Prediction} & \textbf{Memory} \\
    \hline
    OLS (degree $p$) & $O(Md^{2p})$                 & $O(d^p)$            & $O(d^{2p})$     \\
    Ridge            & $O(Md^2)$                    & $O(d)$              & $O(d^2)$        \\
    LASSO            & $O(Md^2)$                    & $O(d)$              & $O(d^2)$        \\
    Random Forests   & $O(M \log (M)  T  \sqrt{d})$ & $O(T  \log (M))$    & $O(T  M)$       \\
    LightGBM         & $O(M  d  T)$                 & $O(T  \log (M))$    & $O(T + b  d)$   \\
    $k$-NN           & $O(Md)$                      & $O(Md + M\log (M))$ & $O(Md)$         \\
    PCA-$k$-NN       & $O(Md^2 + Md)$               & $O(Mk' + M\log M)$  & $O(Mk')$        \\
    Neural Networks  & $O(M  d  h  e)$              & $O(d  h)$           & $O(d  h + h^2)$ \\
    \hline
  \end{tabular}

  \medskip
  \footnotesize
  Note: $M$: samples, $d$: dimension, $T$: trees, $h$: hidden width, $e$: epochs, $p$: polynomial degree, $b$: bins in histogram, $k'$: PCA components
\end{table}

\cref{fig:scaling_all} shows training and prediction times for models using the hyperparameters from our experiments in \cref{sec:numerical}.
\begin{figure*}[htbp]
  \centering
  \includegraphics[width=\textwidth]{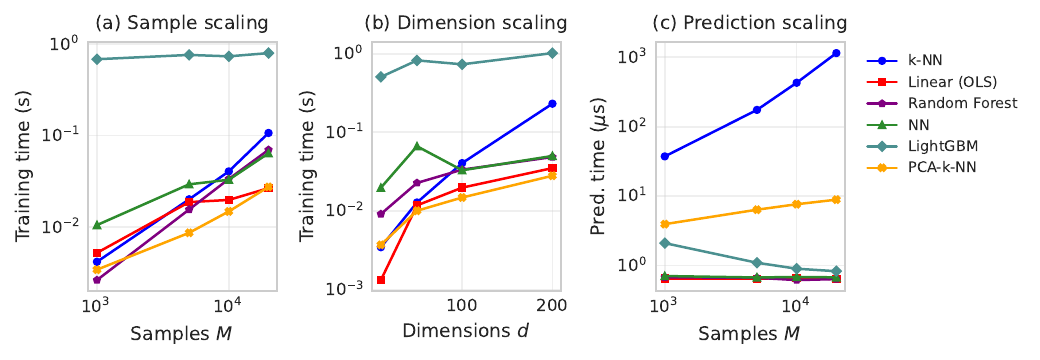}
  \caption{Model scaling analysis: (a) training time vs samples, (b) training time vs dimensions, (c) prediction time. Training times reflect the cost of training one model for a single time step and mode; total algorithm runtime follows the complexity formula in \cref{sec:algorithm}. Our algorithm in \cref{sec:algorithm} required minutes to an hour depending on the specific regression model, hyperparameters, and which experiment we ran.}
  \label{fig:scaling_all}
\end{figure*}

\subsection{Linear models}
All linear models are configured to first standardize the input features to zero mean and unit variance. This improves numerical stability and ensures regularization is applied uniformly across features.

\textbf{Ordinary Least Squares} serves as our baseline, using polynomial degree 6 for $d \leq 2$ to capture nonlinear decision boundaries, and degree 1 otherwise to avoid parameter explosion.

\textbf{Ridge regression} ($\lambda = 0.1$) handles correlated state variables common in optimal switching problems, where for instance asset prices move together. The L2 penalty stabilizes estimates in these ill-conditioned settings without eliminating potentially useful features.

\textbf{LASSO regression} ($\lambda = 0.1$) performs automatic feature selection through L1 penalties, identifying situations where switching decisions depend on few key variables rather than the full state space. This sparsity can improve both interpretability and performance in higher dimensions.

\subsection{Tree-based methods}
\

\textbf{Random forests} (25 trees, depth 3) reduce noise through bootstrap aggregation-averaging predictions across trees trained on different data subsets smooths Monte Carlo error. Shallow trees prevent individual trees from overfitting while ensemble averaging provides stability across the hundreds of models required for optimal switching.

\textbf{LightGBM} (200 iterations, learning rate 0.05) achieves robustness through regularization: L2 penalties, 80\% data sub-sampling, minimum 100 points per leaf. The histogram approach also cuts memory usage in high dimensions.

\subsection{Local regression methods}
\

\textbf{$k$-Nearest Neighbors} ($k=10$) averages over neighborhoods, naturally smoothing Monte Carlo noise without distributional assumptions. 

Standard $k$-NN fails in dimensions $\sim 10$ due to distance concentration. We test PCA-$k$-NN as a dimensionality reduction approach, accepting information loss for meaningful nearest neighbor distances.

\textbf{PCA-$k$-NN} (6 components) projects to principal components before applying $k$-NN.

\subsection{Neural networks}
\

\textbf{Architecture exploration.}  We explored various feedforward neural network configurations to ensure fair comparison, testing ReLU and Tanh activations, depths up to 4 layers, widths up to 512 units, different optimizer parameters (learning rates, decay schedules), with and without dropout, variable-breadth architectures, and both with and without early stopping.

Based on this exploration, we selected architectures up to 128 hidden units and 2 layers, using dropout for regularization and ADAM optimization with exponential learning rate decay. These configurations provided the best performance across experiments while maintaining training times comparable to other methods.

\textbf{Neural networks exploit temporal structure} through backward initialization: parameters for time-step $n$ initialize from time-step $n+1$ with small Gaussian perturbations ($\theta_n = \theta_{n+1} + 0.005 \cdot \mathcal{N}(0,1)$). This leverages the assumption that $\mathcal{R}_{n,j} \approx \mathcal{R}_{n+1,j}$ since continuation values change smoothly over small time-intervals $\Delta t$.

More detailed specifications of architectures, implementation details, and a summary of hyperparameters for the different regression models can be found in \cref{app:implementation}.

\section{Theoretical guarantees for the Longstaff-Schwartz algorithm} \label{sec:theoretical}

In this section we provide theoretical guarantees for the Longstaff-Schwartz algorithm when the regression method is $k$-nearest neighbor regression (foreshadowing the results of our numerical experiments, this regression model also performs well in practice). The theoretical guarantees are one-step concentration bounds between the approximate value function \cref{eq:tildeV}, which uses the previous value-function estimate through the continuation value \cref{eq:gj}, and our estimate \cref{eq:Vhat:theoretical}.

We first treat the case when the one-step transition kernel of the discretized \emph{diffusion} dynamics admits sub-Gaussian (Gaussian-type) bounds, and then we treat a \emph{jump-diffusion} setting where the corresponding bounds are sub-exponential, which in particular covers the jump-diffusion case with one-dimensional sub-exponential jumps. Let us note that the processes in our numerical examples have heavier tails, and that sub-Gaussian and sub-exponential transitions are additional assumptions to develop our theoretical results. See Remark~\ref{remark:moment} on extensions to heavier tailed processes.

The $k$-nearest neighbor estimator at time-step $n$ and mode $j$ is defined by:
\begin{equation}\label{eq:knn}
  R_{n,j}(x) = \frac{1}{k}\sum_{s \in \mathcal{N}_{t_n}(x)} Y_{j,t_n}^s,
\end{equation}
where $\mathcal{N}(x)$ are the indices of the $k$ nearest neighbors of $x$ at time $t_{n}$, in the training set.
\subsection{The sub-Gaussian case}
In our theoretical bounds, we assume that our state process $X_t$ has a transition density that satisfies the following bounds: For a fixed $\Delta t$, there exist constants $\lambda_1,\lambda_2>0$ such that for every $x,y \in \R^d$ we have
\begin{equation} \label{eq:transition:bounds}
  \begin{split}
    |\rho(x,t;y,t+\Delta t)|          & \leq \frac{1}{\lambda_1} \exp(-\lambda_1 \|x-y\|^2),
    \\
    |\nabla_x \rho(x,t;y,t+\Delta t)| & \leq \frac{1}{\lambda_2} \exp(-\lambda_2 \|x-y\|^2).
  \end{split}
\end{equation}
\begin{remark}
  If $X_t$ is a diffusion process, the above estimates hold under rather mild conditions on the drift and diffusion terms. 
\end{remark}

Recall from \cref{sec:regression} that $\mathcal{X} = \{X^{s}_{t_0},\ldots,X^{s}_{t_N}; s=1,\ldots,M\}$ is the set of all states at all times for the $M$ independent trajectories, and that $\widehat{V}_j(t_{n+1},x)$ is an approximation of the value function $V_j(t_{n+1},x)$. The per-mode continuation value we estimate in \cref{alg:main} at each time-step is 
\begin{equation*}
  g_j(t_n,x) = \int_{\R^d}{\widehat V}_j(t_{n+1},y)\rho(x,t_n;y,t_n+\Delta t)\,dy.
\end{equation*}
The approximate value function, which serves as the target in our theoretical analysis, is then
\begin{equation} \label{eq:tildeV}
  \widetilde{V}_i(t_n,x) = \max_{j=1,\dots,D}\left[\Delta t\,f_j(t_n,x)
    - c_{ij}(t_n,x) + g_j(t_n,x)\right].
\end{equation}

Now let $k$ be a positive integer. Then the $k$-NN estimator of $g_j(t_n,x)$, see \cref{eq:knn} for the functional form, uses the data pairs in \cref{eq:regressionPairs} and is given by
\begin{equation} \label{eq:gjtilde}
  \widetilde{g}_j(t_n,x) = \frac{1}{k m_Y} \sum_{s \in \mathcal{N}_{t_n}(x)} \sum_{\hat{s}=1}^{m_Y}
    \widehat{V}_j(t_{n+1}, \widetilde{X}_{t_{n+1}}^{s,\hat{s}}),
\end{equation}
where $(\widetilde{X}_{t_{n+1}}^{s,1},\ldots,\widetilde{X}_{t_{n+1}}^{s,m_Y})$ are one-step transitions as defined in \cref{sec:regression}.

The final estimate at time $t_n$ of the continuation value $g_j(t_n,x)$ is the following truncated version of $\widetilde{g}_j(t_n,x)$,
\begin{equation} \label{eq:truncation}
  \widehat{g}_j(t_n,x) = \mathrm{sign}(\widetilde{g}_j(t_n,x)) \cdot \min\{|\widetilde{g}_j(t_n,x)|, 2 C_0(\|x\|+1)\},
\end{equation}
where $C_0$ is a given constant. The value function is then reconstructed by taking the maximum over modes, as in \cref{eq:Vhat:reconstruction}:
\begin{equation} \label{eq:Vhat:theoretical}
  \widehat{V}_i(t_n,x) = \max_{j=1,\dots,D}\left[\Delta t\,f_j(t_n,x)
    - c_{ij}(t_n,x) + \widehat{g}_j(t_n,x)\right].
\end{equation}

\begin{remark}
  The truncated continuation value $\widehat{g}_j(t_n,x)$ is defined to ensure that the growth condition in the following theorem is always satisfied, although the constant $C_0$ needs to grow as $n$ becomes smaller. Thus, the following theorem gives us the one-step error in supremum norm inside a fixed cube $Q$.
\end{remark}

Before stating our results, let us also make clear the random structure of the problem, with respect to which we make our concentration estimates. There are two sources of randomness to consider: one is from our simulated trajectories $\mathcal{X}$, and the other from the one-step transitions $\widetilde X_{t_{n+1}}^{s,\hat s}$. By construction in \cref{alg:main}, conditional on $\mathcal{X}$, the one-step transitions are mutually independent.

\begin{theorem} \label{thm:1}
  Let \cref{eq:transition:bounds} hold and assume that $|\widehat{V}_i(t_{n+1},x)| \leq L (\|x\|+1)$ for all $i = 1,\dots,D$ and some $L > 0$.
  Then for any fixed bounded set $Q \subset \R^d$ there exist positive constants
  $C_0$ (used in \cref{eq:truncation}), $C_1$, $C_2$ and $\sigma$, depending only on $d$, $L$, $Q$, $\lambda_1$ and $\lambda_2$, such that for any $\delta > 0$ it holds
  \begin{multline*}
    \P(\sup_{x \in Q} |\widehat{V}_i(t_n,x) - \widetilde{V}_i(t_n,x)| \geq \delta) \leq
    2D{\binom{M}{k}} \exp \left (-\frac{k m_Y\delta^2}{18 \sigma^2} \right )
    \\
    + D \cdot
    \begin{cases}
      C_1 \frac{|Q|}{\delta^d} \exp \left (-\frac{M \left (\frac{M-k}{M} - \sup_Q p_{\delta/(4 C_2)}^c \right )_+^2}{2} \right )            & \text{if } \delta/C_2 \leq 1,
      \\
      C_1 \frac{|Q|}{\delta^{d/2}} \exp \left (-\frac{M \left (\frac{M-k}{M} - \sup_Q p_{\sqrt{\delta/(4 C_2)}}^c \right )_+^2}{2} \right ) & \text{if } \delta/C_2 > 1,
    \end{cases}
  \end{multline*}
  where $D$ is the number of modes, $p_{r}^c = \P(X_{t_n} \not \in B(x,r))$, $B(x,r)$ is the Euclidean ball of radius $r$ centered at $x \in \R^d$, and $|Q|$ is the volume of the set $Q$.
\end{theorem}

In the above theorem, certain relations between $k,M$ and $m_Y$ need to hold in order for the estimate to become useful.

\medskip
\noindent\textbf{(i) Neighborhood size: $k$ cannot be too large.}
For a given $\delta>0$ we need
\begin{equation*}
  \frac{M-k}{M} - p_0 > 0 \qquad \Longleftrightarrow \qquad k < M(1-p_0),
\end{equation*}
where $p_0$ corresponds to the two $\delta$-regimes in \cref{thm:1}:
\begin{equation*}
  p_0 = \sup_{x\in Q} p_{\delta/(4C_2)}^c(x)\quad \text{if }\delta/C_2\le 1,
  \qquad\text{and}\qquad
  p_0 = \sup_{x\in Q} p_{\sqrt{\delta/(4C_2)}}^c(x)\quad \text{if }\delta/C_2>1.
\end{equation*}
Intuitively, if $k$ is too large then the $k$-NN estimator oversmooths and cannot track the local structure of the continuation value.

\medskip
\noindent\textbf{(ii) One-step Monte Carlo noise: $m_Y$ must be large enough.}
The noise induced by the one-step simulations is controlled by $m_Y$. While averaging over $k$ neighbors provides additional concentration, we also need the leading concentration term to vanish, i.e.
\begin{equation*}
  {\binom{M}{k}} \exp\!\left(-\frac{k m_Y\delta^2}{18 \sigma^2}\right)\to 0.
\end{equation*}
Using the standard bound ${\binom{M}{k}} \le (eM/k)^k$ (e.g.\ when $k\le M/2$), it suffices that
\begin{equation*}
  \exp\!\left(k\Big(\log(M/k)+1-\frac{m_Y\delta^2}{18\sigma^2}\Big)\right)\to 0.
\end{equation*}
Hence, a sufficient condition is
\begin{equation*}
  m_Y > \frac{18\sigma^2(\log(M/k)+1)}{\delta^2}\quad \text{if }k\to\infty,
\end{equation*}
while if $k$ is fixed, it is enough that $m_Y$ grows strictly faster than $\log M$.

\begin{theorem} \label{thm:2}
  Under the same assumptions of \cref{thm:1} we have that for any fixed point $x \in \R^d$, there exist positive constants $C_0$ (used in \cref{eq:truncation}), $C_2$ and $\sigma$, depending only on $d$, $L$, $|x|$, $\lambda_1$ and $\lambda_2$, such that for any $\delta > 0$ it holds
  \begin{multline*}
    \P (|\widehat{V}_i(t_n,x) - \widetilde{V}_i(t_n,x)| \geq \delta)
    \leq 2D \exp \left (-\frac{k m_Y\delta^2}{18 \sigma^2} \right )
    \\
    +
    2D
    \begin{cases}
      \exp \left (-\frac{M \left (\frac{M-k}{M} - p_{\delta/(4 C_2)}^c \right )_+^2}{2} \right )        & \text{if } \delta/C_2 \leq 1,
      \\
      \exp \left (-\frac{M \left (\frac{M-k}{M} - p_{\sqrt{\delta/(4 C_2)}}^c \right )_+^2}{2} \right ) & \text{if } \delta/C_2 > 1,
    \end{cases}
  \end{multline*}
  where $p_{r}^c = \P(X_{t_n} \not \in B(x,r))$.
\end{theorem}

We note that the above theorem has a better concentration property compared to \cref{thm:1}, specifically we can take $m_Y = 1$ as long as $k \to \infty$ as $M \to \infty$ and the classical assumption that $k/M \to 0$ as $M \to \infty$ is satisfied. In this case, for large enough $M$ we have
\begin{equation*}
  \P (|\widehat{V}_i(t_n,x) - \widetilde{V}_i(t_n,x)| \geq \delta) \leq 4D \exp \left (-\frac{k \delta^2}{18 \sigma^2} \right ).
\end{equation*}

\subsection{The sub-exponential case}
We will here assume that our Markov process has a transition density that satisfies the following bounds: For a fixed $\Delta t$, there exist constants $\lambda_1,\lambda_2>0$ such that for every $x,y \in \R^d$ we have
\begin{equation} \label{eq:transition:bounds:subexp}
  \begin{split}
    |\rho(x,t;y,t+\Delta t)|          & \leq \frac{1}{\lambda_1} \exp(-\lambda_1 \|x-y\|),
    \\
    |\nabla_x \rho(x,t;y,t+\Delta t)| & \leq \frac{1}{\lambda_2} \exp(-\lambda_2 \|x-y\|).
  \end{split}
\end{equation}
\begin{remark}
  In \cref{app:3} we show that the above estimates hold for a class of jump-diffusion processes that only have jumps in one dimension.
\end{remark}

\begin{theorem} \label{thm:3}
  Let \cref{eq:transition:bounds:subexp} hold and assume that $|\widehat{V}_i(t_{n+1},x)| \leq L (\|x\|+1)$ for all $i = 1,\dots,D$ and some $L > 0$.
  Then for any fixed bounded set $Q \subset \R^d$ there exist positive constants
  $C_0$ (used in \cref{eq:truncation}), $C_1$, $C_2$ and $\sigma$, depending only on $d$, $L$, $Q$, $\lambda_1$ and $\lambda_2$, such that for any $\delta > 0$ it holds
  \begin{multline*}
    \P(\sup_{x \in Q} |\widehat{V}_i(t_n,x) - \widetilde{V}_i(t_n,x)| \geq \delta) \leq
    2D{\binom{M}{k}} \max \left \{ \exp \left (-\frac{k m_Y \delta^2}{18 \sigma^2} \right ), \exp \left (-\frac{k m_Y \delta}{6 \sigma} \right )\right \}
    \\
    + D \cdot
    \begin{cases}
      C_1 \frac{|Q|}{\delta^d} \exp \left (-\frac{M \left (\frac{M-k}{M} - \sup_Q p_{\delta/(4 C_2)}^c \right )_+^2}{2} \right )            & \text{if } \delta/C_2 \leq 1,
      \\
      C_1 \frac{|Q|}{\delta^{d/2}} \exp \left (-\frac{M \left (\frac{M-k}{M} - \sup_Q p_{\sqrt{\delta/(4 C_2)}}^c \right )_+^2}{2} \right ) & \text{if } \delta/C_2 > 1,
    \end{cases}
  \end{multline*}
  where $p_{r}^c = \P(X_{t_n} \not \in B(x,r))$.
\end{theorem}
\begin{remark}\label{remark:moment}
  We note that one can extend \cref{thm:3}, using the moment method (see for instance~\cite{boucheron2013concentration}), to the case when \cref{eq:transition:bounds:subexp} is replaced by an upper bound of the form $1/K(\|x-y\|)$ where $K$ is a non-negative polynomial. In this case the decay of the first term on the right-hand side in \cref{thm:3} will depend on the growth of $K$.
\end{remark}

\section{Numerical experiments} \label{sec:numerical}

We test machine learning methods in the Longstaff-Schwartz algorithm across four experiments to evaluate different aspects of method performance: robustness across problem types, multiscale planning, and dimensionality scaling.

Our experimental design uses established examples from~\cite{carmona2008pricing,ACLP12,bayraktar2023neural}, and a new artificial problem that requires multi-timescale planning. All experiments use a 90/10 train/validation split and models, except neural networks, had minimal hyperparameter tuning to assess out-of-the-box performance.

Overview of examples used:
\begin{itemize}
  \item \textbf{Carmona-Ludkovski (CL) and ACLP:} Established examples from the optimal switching literature
  \item \textbf{Artificial planning (BSP):} Designed to test scenarios where myopic strategies should fail due to multi-timescale interactions
  \item \textbf{High-dimensional CL (HCL):} Tests dimensional scaling by extending CL to higher dimensions
\end{itemize}
\subsection{Experiment 1: Carmona-Ludkovski (CL)} \label{sec:numerical:carmona}
The Carmona-Ludkovski experiment models gas-fired power plant management, following the formulation in~\cite{carmona2008pricing} as adjusted by~\cite{bayraktar2023neural}.

The model captures interplay between electricity price $P_t$ and gas price $G_t$, incorporating market features such as mean-reverting dynamics, price correlation, and electricity price spikes.
The state process $X_t=(P_t, G_t)$ combines a jump-diffusion process for electricity price $P_t$ and a diffusion process for gas price $G_t$. In differential form, the dynamics are given by
\begin{equation}
  dX_t = b(X_t) dt + \sigma(X_t) dW_t + dJ_t,
\end{equation}
where $W_t \in \mathbb{R}^2$ is a standard two-dimensional Brownian motion. The jump process $J_t$ affects only the first dimension (electricity price) and is defined as:
\begin{equation}
  J_t = \sum_{i=1}^{N_t} (e^{Y_i} - 1)e_1, \quad N_t \sim \text{Poisson}(32t), \quad Y_i \sim \text{Exponential}(10),
\end{equation}
where $e_1$ denotes the first standard basis vector in $\mathbb{R}^2$. The drift function $b(X_t)$ captures the mean-reverting behavior of the electricity and gas prices:
\begin{equation}
  b(X_t) = \begin{pmatrix}
    5P_t(\log 50 - \log P_t) \\
    2G_t(\log 6 - \log G_t).
  \end{pmatrix}
\end{equation}
The diffusion function incorporates price-dependent volatility and correlation:
\begin{equation}
  \sigma(X_t) = \begin{pmatrix}
    0.5P_t  & 0       \\
    0.32G_t & 0.24G_t
  \end{pmatrix}.
\end{equation}

Managing the power plant, we have three production modes: off, half capacity, and full capacity. The payoff functions for the three modes are:
\begin{equation}
  f_j(t,X_t) = \begin{cases}
    -1,                        & \text{if } j = 1  \\
    0.438(P_t - 7.5G_t) - 1.1, & \text{if } j = 2  \\
    0.876(P_t - 10G_t) - 1.2,  & \text{if } j = 3.
  \end{cases}
\end{equation}
The switching-cost from mode $i$ to mode $j$ (ramping up or down production) is:
\begin{equation}
  c_{ij}(t,X_t) = \begin{cases}
    0,               & \text{if } i = j     \\
    0.01G_t + 0.001, & \text{if } i \neq j.
  \end{cases}
\end{equation}
The model parameters used in our simulations are summarized in \cref{table:carmona}.
\begin{table}[ht]
  \caption{Model parameters for the CL example}
  \label{table:carmona}
  \centering
  \begin{tabular}{@{}llr@{}}
    \toprule
    Parameter                & Value                               & Description                                            \\
    \midrule
    $d$                      & 2                                   & Dimension of random process                            \\
    $D$                      & 3                                   & Number of modes                                        \\
    $M$                      & 50000                               & Number of trajectories                                 \\
    $N$                      & 180                                 & Number of time points                                  \\
    $m_Y$                    & 1                                   & Number of one-step trajectories in training            \\
    $t_\text{start}$         & 0                                   & Start time                                             \\
    $t_\text{end}$           & 0.25                                & End time                                               \\
    $\Delta t$               & $(t_\text{end} - t_\text{start})/N$ & Time step                                              \\
    $\lambda_\text{Poisson}$ & 32                                  & Average number of jumps in electricity prices per year \\
    $\lambda_\text{exp}$     & 10                                  & Rate parameter of exponential jump sizes               \\
    \bottomrule
  \end{tabular}
\end{table}
The initial value $X_0$ is set as:
\begin{equation}\label{eq:carmona:x0}
  X_0 = (50, 6) + \epsilon, \quad \epsilon \sim \mathcal{N}(0, 0.01^2I)
\end{equation}

\begin{remark}
  Our implementation uses a jump intensity of $\lambda_\text{Poisson} = 32$ in the electricity price process, compared to $\lambda_\text{Poisson} = 8$ in the original models of~\cite{carmona2008pricing} and~\cite{bayraktar2023neural}. This modification was motivated by preliminary numerical experiments, which revealed that lower jump intensities resulted in mostly static switching strategies, with most trajectories staying in the same mode through the time horizon. The higher jump frequency induces more frequent transitions between operational modes, allowing for a more thorough examination of the switching boundaries and how the different models perform in these regions.
\end{remark}
\subsection{Experiment 2: A{\"\i}d-Campi-Langrené-Pham (ACLP)}\label{sec:numerical:aid}
The ACLP experiment tests method performance on a higher-dimensional fuel management problem~\cite{ACLP12}, as adapted in~\cite{bayraktar2023neural}, where power plants optimize across multiple fuel types with varying availability and costs. Unlike the two-dimensional CL problem, this 9-dimensional setting challenges methods to handle more complex state interactions while maintaining decision quality.

The model captures a power plant manager choosing between four operational modes, each representing different fuel allocation strategies. The state process $X_t=(Z_t,S_t) \in \mathbb{R}^9$ tracks demand/availability factors ($Z_t$) and various fuel and electricity prices ($S_t$).

In more detail: The state process $X_t=(Z_t,S_t) \in \mathbb{R}^9$ combines two components:
\begin{itemize}
  \item $Z_t = (Z^0_t, Z^1_t, Z^2_t, Z^3_t)$ represents demand and availability rates of three fuel types.
  \item $S_t = (S^0_t, S^1_t, S^2_t, S^3_t, S^4_t)$ represents various prices.
\end{itemize}
These components determine
\begin{itemize}
  \item Electricity demand $D_t = Z^0_t + H^0(t)$,
  \item Availability rates $A_t^\phi = \kappa(Z^\phi_t + H^\phi(t))$ for $\phi = 1,2,3$,
  \item Fuel costs and electricity prices $S_t$,
\end{itemize}
where $H^0(t)$ and $H^\phi(t)$ are seasonal adjustments that ensure demand and availability rates stay within certain bounds, and $\kappa$ is the standard normal cumulative distribution function, setting availability rates between 0 and 1.

The dynamics follow a jump-diffusion process:
\begin{equation}
  \begin{aligned}
    dZ_t & = -\alpha \odot Z_t dt + \beta dW_t^1, \quad \beta \in \mathbb{R}^{4 \times 4}                          \\
    dS_t & = \Xi S_t dt + s_\sigma \text{diag}(S_t) dW_t^2 + S_t \odot dJ_t, \quad \Xi \in \mathbb{R}^{5 \times 5}
  \end{aligned}
\end{equation}
where $\odot$ denotes the Hadamard product (element-wise multiplication), $\alpha \in \mathbb{R}^{4}$ is a mean-reversion vector, $\beta$ represents volatility, and $\Xi$ captures price correlations and $W_t = (W^1_t,W^2_t)$ is a $9$-dimensional Brownian motion. The jump process affects only the price components and follows:
\begin{equation}
  J_t^i =
  \sum_{k=1}^{N_t^i} (e^{Y_k^i}-1), \quad N_t^i \sim \text{Poisson}(\lambda_{\text{Poisson}} t), \quad Y_k^i \sim \text{Exponential}(\lambda_{\text{exp}}), \quad i = 1,\dots,5.
\end{equation}

Now we define the payoff functions and switching costs. The plant manager operates in four distinct modes, each representing a specific allocation of capacity across different fuel sources. These modes are encoded as rows in a capacity matrix $C \in \mathbb{R}^{4 \times 3}$, where $C_{j,k}$ represents the capacity allocated to fuel type $k$ in mode $j$.
The payoff function $f_j: \mathbb{R}^9 \times [0,1] \to \mathbb{R}$ for mode $j$ is given by:
\begin{equation}
  f_j(t,X_t) = \min \left\{\sum_{k=1}^3 C_{j,k}A_t^k, D_t\right \} S^4_t - \sum_{k=1}^3 C_{j,k}(h_{\text{CO2},k}S^0_t + h_{\text{tech},k}S^k_t)
\end{equation}
where $C_{j,k}A_t^k$ represents the effective capacity for fuel type $k$ in mode $j$, adjusted by its availability factor. The first term calculates revenue as the minimum of total effective capacity and demand, multiplied by the electricity price. The second term represents operational costs, including both carbon dioxide emissions and technology-specific costs.

The cost of switching from mode $i$ to mode $j$ is defined by:
\begin{equation}
  c_{ij}(t,X_t) = \frac{1}{3} \sum_{k=1}^3 S^k_t \cdot \1_{C_{i,k} \neq C_{j,k}} + 0.001
\end{equation}
where $\1_{C_{i,k} \neq C_{j,k}}$ is the indicator function that equals 1 when the capacity allocation for fuel $k$ differs between modes $i$ and $j$, and 0 otherwise. This cost reflects the expenses of fuel replacement during mode transitions, proportional to current fuel prices, plus a small fixed cost.

The starting distribution is
\begin{equation}\label{eq:aid:x0}
  (70,1,1,0,20,60,40,20,120) + \epsilon, \quad \epsilon \sim \mathcal{N}(0,0.005^2I)
\end{equation}

The exact values of the capacity matrix $C$, and other parameters $h_{\text{CO2},k}$, $h_{\text{tech},k}$, $\beta$, $s_\text{long}$, $s_\alpha$, and $s_\sigma$ can be found in the code~\cite{code}, and the main experiment parameters are found in \cref{tab:aid_params}.
\begin{table}[ht]
  \centering
  \caption{Model parameters for the ACLP example}
  \label{tab:aid_params}
  \begin{tabular}{@{}llr@{}}
    \toprule
    Parameter                & Value                               & Description                                            \\
    \midrule
    $d$                      & 9                                   & Dimension of random process                            \\
    $N$                      & 90                                  & Number of time steps                                   \\
    $D$                      & 4                                   & Number of modes                                        \\
    $M$                      & 50000                               & Number of trajectories                                 \\
    $m_Y$                    & 1                                   & Number of one-step trajectories in training            \\
    $t_\text{start}$         & 0                                   & Start time                                             \\
    $t_\text{end}$           & 1                                   & End time                                               \\
    $\Delta t$               & $(t_\text{end} - t_\text{start})/N$ & Time step                                              \\
    $\lambda_\text{Poisson}$ & 12                                  & Average number of jumps in electricity prices per year \\
    $\lambda_\text{exp}$     & 15                                  & Rate parameter of exponential jump sizes               \\
    \bottomrule
  \end{tabular}
\end{table}

\subsection{Experiment 3: Banded Shift Process (BSP)}\label{sec:numerical:stapel}

The BSP experiment tests whether methods can learn forward-planning strategies that outperform more greedy approaches. We design a one-dimensional process where optimal modes form spatial bands that shift over time, requiring incremental transitions rather than immediate jumps to maximize long-term value.

The state process $X_t \in \mathbb{R}$ follows mean-reverting dynamics with time-varying mean:
\begin{equation}
  dX_t = -\alpha(X_t - m(t))dt + \sigma(X_t)dW_t
\end{equation}
where $W_t$ is a standard Brownian motion. The model components are defined as follows:
\begin{align}
  m(t)      & = \sin(2\pi t)   &  & \text{(time-varying mean)}          \\
  \alpha      & = 0.5            &  & \text{(mean reversion speed)}       \\
  \sigma(X_t) & = 0.5 + 0.2|X_t| &  & \text{(state-dependent volatility)}
\end{align}
The controller chooses between 10 operational modes in a setting designed to punish myopic decisions. We divide the state space into adjacent bands, where each mode achieves maximum payoff (1.0) in its designated band and reduced payoff (0.2) in neighboring bands.

Specifically, the state range $[-2, 2]$ is partitioned into bands of width 0.4, with mode $j$ optimal when $X_t$ falls in band $j$. Adjacent modes receive moderate payoffs, creating a local neighborhood structure. The payoff function $f_j : \mathbb{R} \times [0,1] \to \mathbb{R}$ is:
\begin{equation}
  f_j(t,X_t) =
  \begin{cases}
    1.0, & \text{if } j = \left\lfloor\frac{X_t + 2}{0.4}\right\rfloor + 1 \text{ and } -2 \leq X_t \leq 2                              \\
    0.2, & \text{if } j = \left\lfloor\frac{X_t + 2}{0.4}\right\rfloor \text{ or } j = \left\lfloor\frac{X_t + 2}{0.4}\right\rfloor + 2 \\
    0,   & \text{otherwise}
  \end{cases}
\end{equation}

Switching costs increase with mode distance but plateau, creating incentives for gradual transitions:
\begin{equation}
  c_{ij}(t,X_t) =
  \begin{cases}
    0,                   & \text{if } i = j    \\
    0.05 \min(|i-j|, 3), & \text{if } i \neq j
  \end{cases}
\end{equation}

The combination of banded payoffs and distance-dependent switching costs creates a challenging control problem that should be difficult for greedy algorithms. For instance, as $X_t$ changes, a greedy controller would immediately shift to the new highest payoff (minus switching cost) mode, whereas an optimal controller might make incremental transitions, temporarily accepting lower rewards to maximize cumulative profits. Furthermore, the time-varying mean and state-dependent volatility introduce additional complexities, requiring the controller to adapt and anticipate changing dynamics.
The initial distribution of $X_0$ is set as:
\begin{equation}\label{eq:stapel:x0}
  X_0 = 0 + \epsilon, \quad \epsilon \sim \mathcal{N}(0, 0.5^2)
\end{equation}

The experiment parameters are summarized in \cref{tab:bsp_params}.
\begin{table}[!hbtp]
  \centering
  \caption{Model parameters for the BSP experiment}
  \label{tab:bsp_params}
  \begin{tabular}{@{}llr@{}}
    \toprule
    Parameter        & Value                               & Description                                 \\
    \midrule
    $d$              & 1                                   & Dimension of random process                 \\
    $D$              & 10                                  & Number of modes                             \\
    $M$              & 20000                               & Number of trajectories                      \\
    $N$              & 36                                  & Number of time points                       \\
    $m_Y$            & 1                                   & Number of one-step trajectories in training \\
    $t_\text{start}$ & 0                                   & Start time                                  \\
    $t_\text{end}$   & 1                                   & End time                                    \\
    $\Delta t$       & $(t_\text{end} - t_\text{start})/N$ & Time step                                   \\
    \bottomrule
  \end{tabular}
\end{table}

\subsection{Experiment 4: High dimensional extension of Carmona-Ludkovski (HCL)}\label{sec:numerical:carmona_dim}

The HCL experiment tests method performance as state dimensions increase from 2 to 50, extending the CL power plant model~\cite{bayraktar2023neural} to higher-dimensional settings. This experiment isolates dimensional scaling effects by maintaining the same three-mode structure while adding correlated state variables.

The high-dimensional state process $X_t \in \mathbb{R}^d$ preserves the jump-diffusion structure of the original CL model, with electricity price jumps affecting only the first dimension and mean-reverting dynamics for all components:
\begin{equation}
  dX_t = b(X_t) dt + \sigma(X_t) dW_t + dJ_t,
\end{equation}
where $W_t \in \mathbb{R}^d$ is a standard $d$-dimensional Brownian motion, and $J_t$ is a jump process only affecting the first dimension (price of electricity). The drift and diffusion functions are given by
\begin{align}
  b(X_t)    & = \boldsymbol{\alpha} \odot X_t \odot (\log(X_0) - \log(X_t)), \\
  \sigma(X_t) & = \Sigma
\end{align}
where $\log$ is applied element-wise,  $\boldsymbol{\alpha} = (5, 2, 2, \ldots, 2) \in \mathbb{R}^d$, and $\Sigma$ is the constant matrix
\begin{equation}
  \Sigma =
  \begin{bmatrix}
    0.5    & 0      & 0      & \cdots & 0      \\
    0.32   & 0.24   & 0      & \cdots & 0      \\
    0.32   & 0      & 0.24   & \cdots & 0      \\
    \vdots & \vdots & \vdots & \ddots & \vdots \\
    0.32   & 0      & 0      & \cdots & 0.24
  \end{bmatrix}.
\end{equation}
There is thus state-dependent drift, and correlation between prices in this model.
The jump process $J_t$ is defined as:
\begin{equation}
  J_t = \sum_{i=1}^{N_t} (e^{Y_i} - 1)e_1, \quad N_t \sim \text{Poisson}(\lambda_\text{Poisson} t), \quad Y_i \sim \text{Exponential}(\lambda_\text{exp})
\end{equation}
where $e_1$ is the unit vector in the first dimension.

The power plant operation maintains the three modes in the original CL problem, each representing differing capacities of the power plant. The plant is off (mode 1), at half capacity (mode 2) or at full capacity (mode 3).
The payoff function $f_j : \mathbb{R}^d \times [0,0.25] \to \mathbb{R}$, $j=1,2,3$, for mode $j$ is:
\begin{equation}
  f_j(t,X_t) = a_j X_{1,t} + b_j \widebar{X}_{2:d,t} + c_j
\end{equation}
where $\widebar{X}_{2:d,t} = \frac{1}{d-1} \sum_{i=2}^d X_{i,t}$ is the mean of all dimensions except the first one. The constants $a_j$, $b_j$, and $c_j$ are defined as:
\begin{equation}
  \begin{bmatrix}
    a_1 & a_2 & a_3 \\
    b_1 & b_2 & b_3 \\
    c_1 & c_2 & c_3
  \end{bmatrix} =
  \begin{bmatrix}
    0  & 0.438  & 0.876 \\
    0  & -3.285 & -8.76 \\
    -1 & -1.1   & -1.2
  \end{bmatrix}.
\end{equation}
The cost of switching from mode $i$ to mode $j$ is:
\begin{equation}
  c_{ij}(t,X_t) =
  \begin{cases}
    0,     & \text{if } i = j    \\
    0.011, & \text{if } i \neq j
  \end{cases}
\end{equation}
The initial value $X_0$ is defined as:
\begin{equation}
  X_0 = (50,6,6,\dots,6)
\end{equation}

The experiment parameters can be seen in \cref{tab:carmona_dim_params}.
\begin{table}[ht]
  \centering
  \caption{Model parameters for the HCL experiment}
  \label{tab:carmona_dim_params}
  \begin{tabular}{@{}llr@{}}
    \toprule
    Parameter                & Value                               & Description                                 \\
    \midrule
    $d$                      & variable                            & Dimension of random process                 \\
    $D$                      & 3                                   & Number of modes                             \\
    $M$                      & 20000                               & Number of trajectories                      \\
    $N$                      & 180                                 & Number of time points                       \\
    $m_Y$                    & 1                                   & Number of one-step trajectories in training \\
    $t_\text{start}$         & 0                                   & Start time                                  \\
    $t_\text{end}$           & 0.25                                & End time                                    \\
    $\Delta t$               & $(t_\text{end} - t_\text{start})/N$ & Time step                                   \\
    $\lambda_\text{Poisson}$ & 32                                  & Poisson process intensity                   \\
    $\lambda_\text{exp}$     & 10                                  & Rate parameter of exponential jump sizes    \\
    \bottomrule
  \end{tabular}
\end{table}

\section{Results of numerical experiments}\label{sec:results}
\subsection{Model performance results}

We evaluate performance mainly using the three metrics: Decision quality $Q^\mathcal{R}$, value capture $\kappa^\mathcal{R}$, and internal consistency $C^\mathcal{R}$ (detailed in \cref{sec:metrics}).

\textbf{$k$-NN achieves robust near-optimal decisions across all problem types.} Despite its simplicity, $k$-NN consistently ranks among top performers for decision quality and value capture across all experiments. In higher dimensions, PCA-$k$-NN maintains this performance while more complex methods deteriorate significantly (\cref{fig:carmona_dim:results_2,fig:combined_results}). This robustness extends across very different problem structures, from simple to complex switching regimes.

\textbf{High-quality decisions without accurate value function approximation.} The high-dimensional Carmona-Ludkovski (HCL) experiment reveals a disconnect between prediction accuracy and decision quality. While $k$-NN achieves high quality decisions and almost full value capture, its internal consistency lags significantly behind linear models and random forests (\cref{fig:carmona_dim_results}).

\textbf{$k$-NN and LGBM accurately capture optimal decision regions.}
In the CL experiment, where the two-dimensional state space allows visualization of decision boundaries (\cref{fig:carmona_boundaries}), $k$-NN and LGBM produce boundaries that closely match the a posteriori optimal regions across time-steps. Neural networks, by contrast, exhibit unstable boundaries that shift substantially between time-steps and at times deviate far from the optimal regions.

\textbf{Training difficulty concentrates in the middle time horizon.} Loss curves (\cref{fig:loss_curves}) consistently show that regression becomes most challenging in intermediate time steps, with easier fitting at both initial conditions and terminal states.

\begin{figure}[!htbp]
  \centering

  \begin{subfigure}[t]{0.48\textwidth}
    \centering
    \begin{tikzpicture}
      \begin{axis}[
          height=0.8\textwidth,
          width=\textwidth,
          ylabel={Value},
          xtick={1,2,3},
          xticklabels={$Q^\mathcal{R}$, $C^\mathcal{R}$, $\kappa^\mathcal{R}$},
          xmin=0.5, xmax=3.5,
          ymin=0.6, ymax=1.05,
          ytick={0.6,0.8,1.0},
          grid=both,
        ]

        \addplot[blue, mark=*, thick] coordinates {
            (1, 0.94547) (2, 0.92753) (3, 0.9949)
          };

        \addplot[orange, mark=diamond*, thick] coordinates {
            (1, 0.94547) (2, 0.95386) (3, 0.9937)
          };

        \addplot[green!60!black, mark=triangle*, thick] coordinates {
            (1, 0.92193) (2, 0.97155) (3, 0.9852)
          };

        \addplot[purple, mark=pentagon*, thick] coordinates {
            (1, 0.75486) (2, 0.92304) (3, 0.9056)
          };

        \addplot[red, mark=square*, thick] coordinates {
            (1, 0.74530) (2, 0.96006) (3, 0.9084)
          };

        \addplot[brown, mark=star, thick] coordinates {
            (1, 0.07895) (2, 1.00000) (3, 0.1960)
          };

      \end{axis}
    \end{tikzpicture}
    \caption{CL Experiment}
    \label{fig:carmona_results}
  \end{subfigure}
  \hfill
  \begin{subfigure}[t]{0.48\textwidth}
    \centering
    \begin{tikzpicture}
      \begin{axis}[
          height=0.8\textwidth,
          width=\textwidth,
          ylabel={Value},
          xtick={1,2,3},
          xticklabels={$Q^\mathcal{R}$, $C^\mathcal{R}$, $\kappa^\mathcal{R}$},
          xmin=0.5, xmax=3.5,
          ymin=0.6, ymax=1.05,
          ytick={0.6,0.8,1.0},
          grid=both,
        ]

        \addplot[blue, mark=*, thick] coordinates {
            (1, 0.9551) (2, 0.7782) (3, 0.9952)
          };

        \addplot[red, mark=square*, thick] coordinates {
            (1, 0.9031) (2, 0.9668) (3, 0.9800)
          };

        \addplot[green!60!black, mark=triangle*, thick, dashed] coordinates {
            (1, 0.9386) (2, 0.7006) (3, 0.9828)
          };

        \addplot[purple, mark=pentagon*, thick] coordinates {
            (1, 0.8953) (2, 0.9815) (3, 0.9732)
          };

        \addplot[orange, mark=diamond*, thick] coordinates {
            (1, 0.8609) (2, 0.8114) (3, 0.8818)
          };

        \addplot[brown, mark=star, thick] coordinates {
            (1, 0.4289) (2, 1.0000) (3, 0.7976)
          };

      \end{axis}
    \end{tikzpicture}
    \caption{HCL Experiment ($d=50$)}
    \label{fig:carmona_dim_results}
  \end{subfigure}

  \vspace{1em}
  \begin{subfigure}[t]{0.48\textwidth}
    \centering
    \begin{tikzpicture}
      \begin{axis}[
          height=0.8\textwidth,
          width=\textwidth,
          ylabel={Value},
          xtick={1,2,3},
          xticklabels={$Q^\mathcal{R}$, $C^\mathcal{R}$, $\kappa^\mathcal{R}$},
          xmin=0.5, xmax=3.5,
          ymin=0.55, ymax=1.05,
          ytick={0.6,0.8,1.0},
          grid=both,
        ]

        \addplot[cyan, mark=diamond*, thick] coordinates {
            (1, 0.64429) (2, 0.98019) (3, 0.9934)
          };

        \addplot[red, mark=square*, thick] coordinates {
            (1, 0.64824) (2, 0.98575) (3, 0.9934)
          };

        \addplot[blue, mark=*, thick] coordinates {
            (1, 0.68308) (2, 0.98140) (3, 0.9956)
          };

        \addplot[brown, mark=star, thick] coordinates {
            (1, 0.66066) (2, 1.00000) (3, 0.9930)
          };

        \addplot[green!60!black, mark=triangle*, thick] coordinates {
            (1, 0.52538) (2, 0.90423) (3, 0.9874)
          };

        \addplot[purple, mark=pentagon*, thick] coordinates {
            (1, 0.56033) (2, 0.96927) (3, 0.9796)
          };

        \addplot[magenta, mark=otimes*, thick] coordinates {
            (1, 0.67330) (2, 0.98750) (3, 0.9943)
          };

        \addplot[orange, mark=diamond*, thick] coordinates {
            (1, 0.63319) (2, 0.91798) (3, 0.9476)
          };

      \end{axis}
    \end{tikzpicture}
    \caption{ACLP Experiment}
    \label{fig:aid_results}
  \end{subfigure}
  \hfill
  \begin{subfigure}[t]{0.48\textwidth}
    \centering
    \begin{tikzpicture}
      \begin{axis}[
          height=0.8\textwidth,
          width=\textwidth,
          ylabel={Value},
          xtick={1,2,3},
          xticklabels={$Q^\mathcal{R}$,$C^{\mathcal{R}}$, $\kappa^\mathcal{R}$},
          xmin=0.5, xmax=3.5,
          ymin=0.4, ymax=1.05,
          ytick={0.4,0.6,0.8,1.0},
          grid=both,
        ]

        \addplot[blue, mark=*, thick] coordinates {
            (1, 0.84944) (2, 0.97934) (3, 0.8721)
          };

        \addplot[green!60!black, mark=triangle*, thick, dashed] coordinates {
            (1, 0.78444) (2, 0.87785) (3, 0.8461)
          };

        \addplot[purple, mark=pentagon*, thick] coordinates {
            (1, 0.77833) (2, 0.78945) (3, 0.8308)
          };

        \addplot[orange, mark=diamond*, thick] coordinates {
            (1, 0.84472) (2, 0.99762) (3, 0.8796)
          };

        \addplot[red, mark=square*, thick] coordinates {
            (1, 0.45139) (2, 0.79773) (3, 0.5468)
          };

        \addplot[brown, mark=star, thick] coordinates {
            (1, 0.00000) (2, 1.00000) (3, 0.0010)
          };

      \end{axis}
    \end{tikzpicture}
    \caption{BSP Experiment}
    \label{fig:stapel_results}
  \end{subfigure}

  \vspace{1em}
  \centering
  \begin{tikzpicture}
    \begin{axis}[
        hide axis,
        scale only axis,
        height=0cm,
        width=0cm,
        legend style={draw=none, legend columns=5, legend cell align=left, font=\footnotesize},
        legend pos=outer north east,
      ]
      \addplot[blue, mark=*, thick] {0};
      \addlegendentry{$k$-NN/PCA $k$-NN}

      \addplot[orange, mark=diamond*, thick] {0};
      \addlegendentry{LGBM}

      \addplot[green!60!black, mark=triangle*, thick] {0};
      \addlegendentry{network 1}

      \addplot[green!60!black, mark=triangle*, thick, dashed] {0};
      \addlegendentry{network 2}

      \addplot[purple, mark=pentagon*, thick] {0};
      \addlegendentry{forest}

      \addplot[red, mark=square*, thick] {0};
      \addlegendentry{linear}

      \addplot[brown, mark=star, thick] {0};
      \addlegendentry{greedy}

      \addplot[cyan, mark=diamond*, thick] {0};
      \addlegendentry{lasso}

      \addplot[magenta, mark=otimes*, thick] {0};
      \addlegendentry{ridge}
    \end{axis}
  \end{tikzpicture}

  \caption{Performance metrics across different strategies for four experiments when starting in state $1$: (a) CL experiment (\cref{sec:numerical:carmona}), (b) High-dimensional CL experiment with $d=50$ (\cref{sec:numerical:carmona_dim}), (c) ACLP experiment (\cref{sec:numerical:aid}), and (d) BSP experiment (\cref{sec:numerical:stapel}). Starting values for paths were generated according to the same distributions as in the respective equations of $X_0$ for each experiment. For Internal Consistency and Value Capture, we have taken the mean at time $t_0$ over all paths. Total paths generated for each experiment was 1000, except for HCL with $d\in \{40,50\}$ where we used 200 paths. Among the linear models, we have only plotted the best performing one.}
  \label{fig:combined_results}
\end{figure}
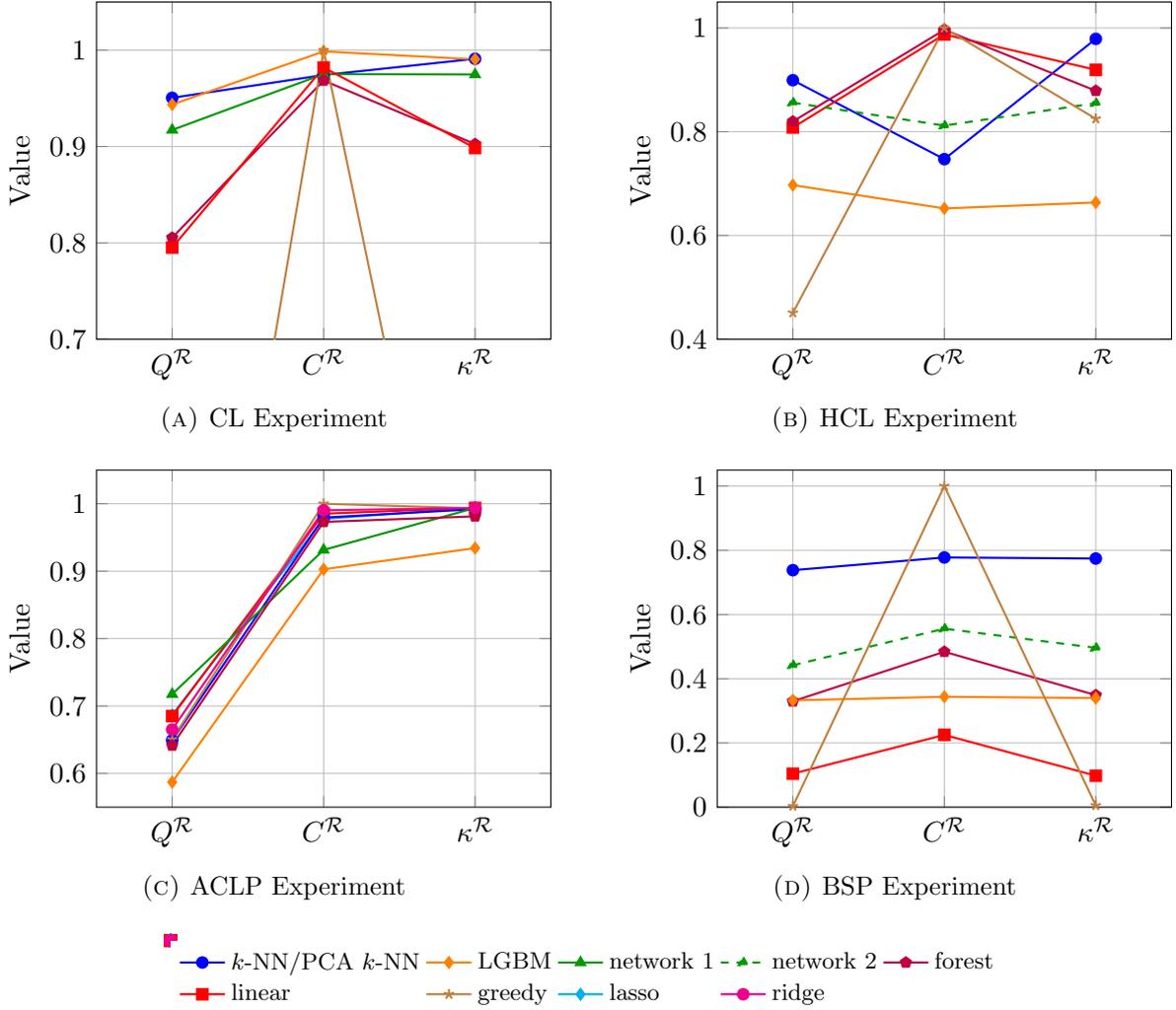

\begin{figure}[!htbp]
  \centering
  \begin{subfigure}[b]{0.45\textwidth}
    \begin{tikzpicture}
      \begin{axis}[
          width=0.9\textwidth,
          height=0.7\textwidth,
          xlabel={Dimension},
          ylabel={Final Value Capture},
          title={Final Value Capture by Dimension},
          grid=both,
          legend style={at={(0.5,-0.22)}, anchor=north, legend columns=3},
          xtick={2,10,20,30,40,50},
          xmin=0, xmax=55,
          ymin=0.8, ymax=1.05,
        ]

        \addplot[blue, thick, mark=*] coordinates {
            (2, 0.99938) (10, 0.99520) (20, 0.99584) (30, 0.99483) (40, 0.99571) (50, 0.99520)
          };

        \addplot[red, thick, mark=square*] coordinates {
            (2, 0.98402) (10, 0.98281) (20, 0.98455) (30, 0.97999) (40, 0.98501) (50, 0.98005)
          };

        \addplot[green!60!black, thick, mark=triangle*] coordinates {
            (2, 0.93137) (10, 0.98500) (20, 0.95859) (30, 0.97965) (40, 0.99516) (50, 0.98278)
          };

        \addplot[purple, thick, mark=pentagon*] coordinates {
            (2, 0.97014) (10, 0.97593) (20, 0.96897) (30, 0.96760) (40, 0.96551) (50, 0.97321)
          };

        \addplot[cyan!70!black, thick, mark=diamond*] coordinates {
            (2, 0.99582) (10, 0.94833) (20, 0.92250) (30, 0.90724) (40, 0.88657) (50, 0.88183)
          };

        \addplot[brown, thick, mark=star] coordinates {
            (2, 0.82513) (10, 0.81500) (20, 0.81083) (30, 0.83456) (40, 0.81716) (50, 0.79755)
          };

        \addplot[orange, thick, mark=otimes*] coordinates {
            (2, 0.98208) (10, 0.98084) (20, 0.98122) (30, 0.97828) (40, 0.98164) (50, 0.97907)
          };

        \addplot[magenta, thick, mark=diamond] coordinates {
            (2, 0.94929) (10, 0.94988) (20, 0.95203) (30, 0.94578) (40, 0.95072) (50, 0.95675)
          };

      \end{axis}
    \end{tikzpicture}
    \caption{Final Value Capture performance across different dimensions}
  \end{subfigure}
  \hfill
  \begin{subfigure}[b]{0.45\textwidth}
    \begin{tikzpicture}
      \begin{axis}[
          width=0.9\textwidth,
          height=0.7\textwidth,
          xlabel={Dimension},
          ylabel={Decision Quality},
          title={Decision Quality by Dimension},
          grid=both,
          legend style={at={(0.5,-0.22)}, anchor=north, legend columns=3},
          xtick={2,10,20,30,40,50},
          xmin=0, xmax=55,
          ymin=0.8, ymax=1.05,
        ]

        \addplot[blue, thick, mark=*] coordinates {
            (2, 0.98523) (10, 0.95122) (20, 0.95786) (30, 0.94318) (40, 0.95050) (50, 0.95506)
          };

        \addplot[red, thick, mark=square*] coordinates {
            (2, 0.90909) (10, 0.90438) (20, 0.91481) (30, 0.88552) (40, 0.90873) (50, 0.90307)
          };

        \addplot[green!60!black, thick, mark=triangle*] coordinates {
            (2, 0.80760) (10, 0.91476) (20, 0.92724) (30, 0.91859) (40, 0.96218) (50, 0.93865)
          };

        \addplot[purple, thick, mark=pentagon*] coordinates {
            (2, 0.92161) (10, 0.90023) (20, 0.88601) (30, 0.86428) (40, 0.86599) (50, 0.89533)
          };

        \addplot[cyan!70!black, thick, mark=diamond*] coordinates {
            (2, 0.97097) (10, 0.91362) (20, 0.90702) (30, 0.86981) (40, 0.85986) (50, 0.86088)
          };

        \addplot[brown, thick, mark=star] coordinates {
            (2, 0.46256) (10, 0.44110) (20, 0.43580) (30, 0.46928) (40, 0.44320) (50, 0.42892)
          };

        \addplot[orange, thick, mark=otimes*] coordinates {
            (2, 0.90324) (10, 0.89942) (20, 0.90656) (30, 0.88113) (40, 0.89961) (50, 0.90127)
          };

        \addplot[magenta, thick, mark=diamond] coordinates {
            (2, 0.84204) (10, 0.84081) (20, 0.85376) (30, 0.82199) (40, 0.84323) (50, 0.86323)
          };

      \end{axis}
    \end{tikzpicture}
    \caption{Decision quality across different dimensions}
    \label{fig:decision_similarity}
  \end{subfigure}
  \vspace{1em}  
  \begin{subfigure}[b]{0.45\textwidth}
    \begin{tikzpicture}
      \begin{axis}[
          width=0.9\textwidth,
          height=0.7\textwidth,
          xlabel={Dimension},
          ylabel={Internal consistency},
          title={Internal consistency by dimension},
          grid=both,
          legend style={at={(0.5,-0.3)},draw=none, anchor=north,legend cell align=left, legend columns=4},
          xtick={2,10,20,30,40,50},
          xmin=0, xmax=55,
          ymin=0.45, ymax=1.05,
        ]

        \addplot[blue, thick, mark=*] coordinates {
            (2, 0.92814) (10, 0.75030) (20, 0.75775) (30, 0.73538) (40, 0.76250) (50, 0.77817)
          };

        \addplot[red, thick, mark=square*] coordinates {
            (2, 0.95377) (10, 0.96973) (20, 0.98530) (30, 0.92451) (40, 0.98233) (50, 0.96678)
          };

        \addplot[green!60!black, thick, mark=triangle*] coordinates {
            (2, 0.51471) (10, 0.79083) (20, 0.87362) (30, 0.75708) (40, 0.92697) (50, 0.70055)
          };

        \addplot[purple, thick, mark=pentagon*] coordinates {
            (2, 0.95875) (10, 0.93225) (20, 0.94250) (30, 0.89802) (40, 0.95886) (50, 0.98151)
          };

        \addplot[cyan!70!black, thick, mark=diamond*] coordinates {
            (2, 0.98955) (10, 0.94895) (20, 0.89374) (30, 0.95130) (40, 0.82854) (50, 0.81141)
          };

        \addplot[brown, thick, mark=star] coordinates {
            (2, 1.00000) (10, 1.00000) (20, 1.00000) (30, 1.00000) (40, 1.00000) (50, 1.00000)
          };

        \addplot[orange, thick, mark=otimes*] coordinates {
            (2, 0.95352) (10, 0.95827) (20, 0.97504) (30, 0.92703) (40, 0.99319) (50, 0.97966)
          };

        \addplot[magenta, thick, mark=diamond] coordinates {
            (2, 0.92411) (10, 0.92580) (20, 0.94207) (30, 0.90186) (40, 0.95677) (50, 0.97450)
          };

        \legend{KNN, linear, network 2, forest, LGBM, greedy, ridge, lasso}
      \end{axis}
    \end{tikzpicture}
    \caption{Internal consistency across different dimensions. Values have been averaged over 1000 paths for each dimension.}
  \end{subfigure}
  \caption{This plot shows how our three measures scale with dimension in \cref{sec:numerical:carmona_dim} when starting in state $1$. Note that for $d>2$, $k$-NN is actually PCA-$k$-NN.}
  \label{fig:carmona_dim:results_2}
\end{figure}
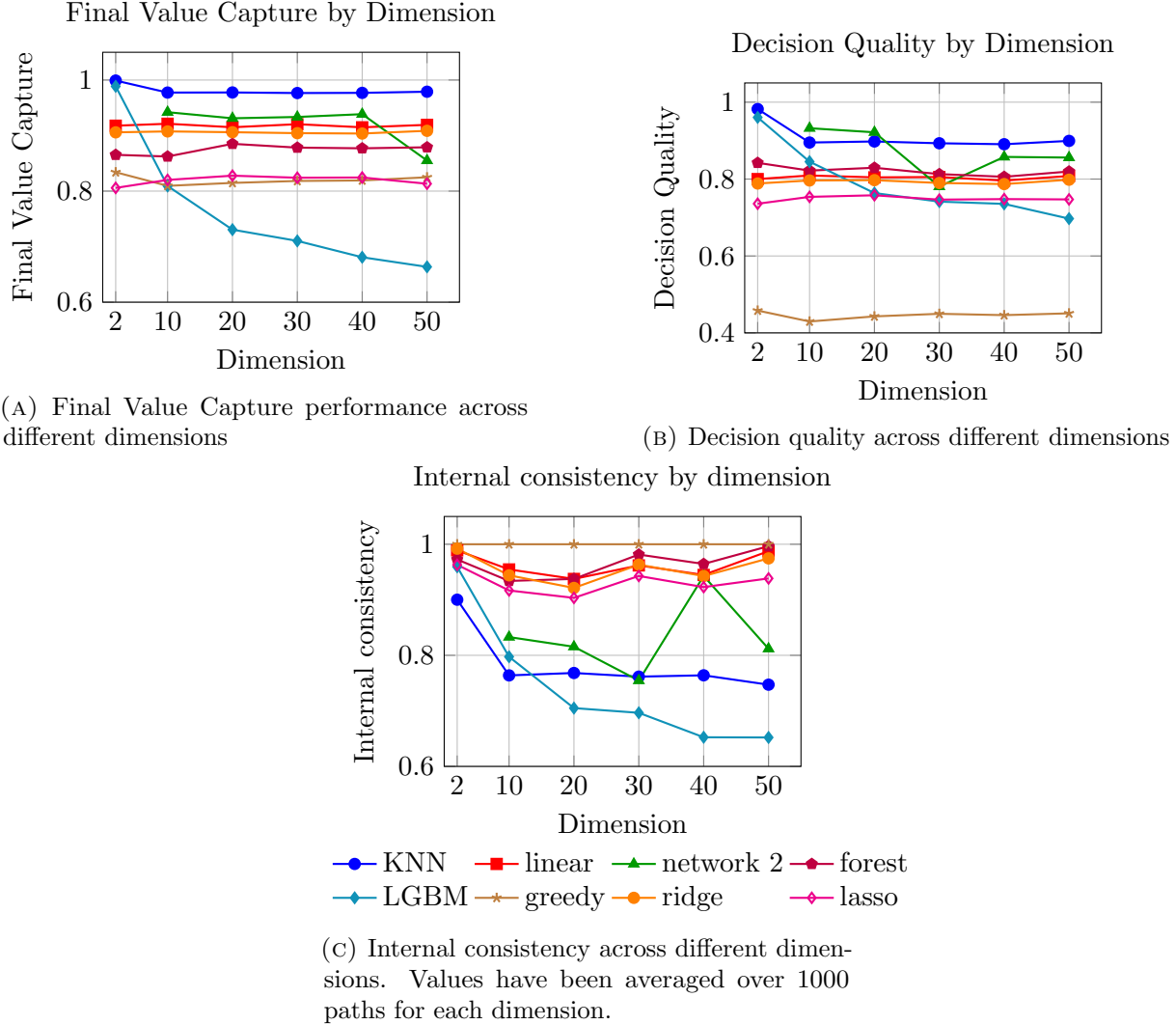
\begin{figure}[!htbp]
  \centering
  \begin{minipage}{0.48\textwidth}
    \centering
    \includegraphics[width=\textwidth]{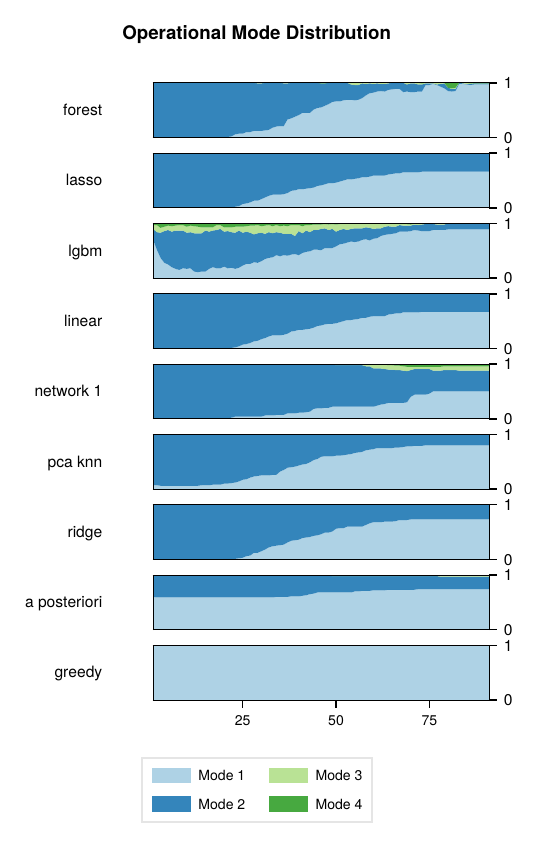}
    \subcaption{Comparison of switching strategies}
    \label{fig:aid-switching-strategies}
  \end{minipage}%
  \hfill
  \begin{minipage}{0.48\textwidth}
    \centering
    \includegraphics[width=\textwidth]{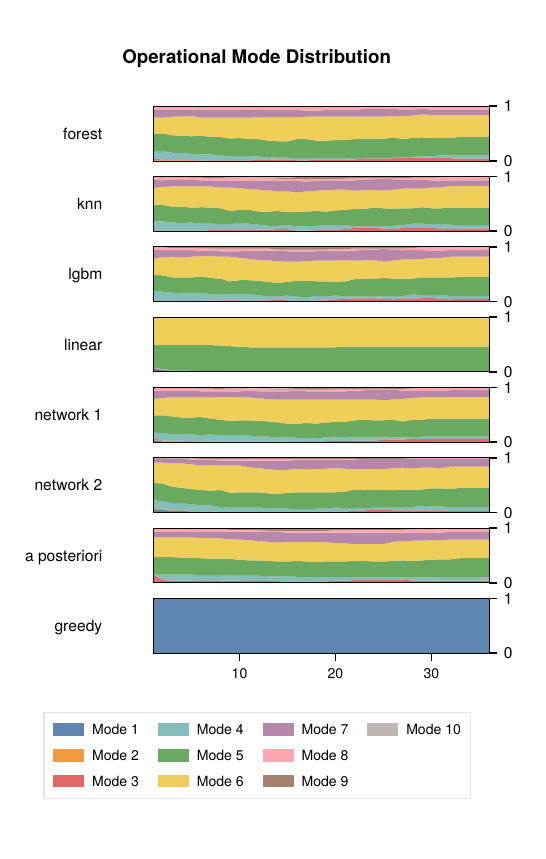}
    \subcaption{Comparison of switching strategies}
    \label{fig:stapel-switching-strategies}
  \end{minipage}
  \caption{Analysis of switching strategies for OS problems. We have simulated 1000 paths and applied the strategies defined by $\mu^{\mathcal{R}}$ for each path and model when starting in state $1$, and then calculated the distribution of strategies across modes over time. Left: ACLP performance. Right: BSP}
  \label{fig:switching-strategies-comparison}
\end{figure}

\begin{figure}[!htbp]
  \centering
  \begin{subfigure}[b]{0.45\textwidth}
    \centering
    \includegraphics[width=\textwidth]{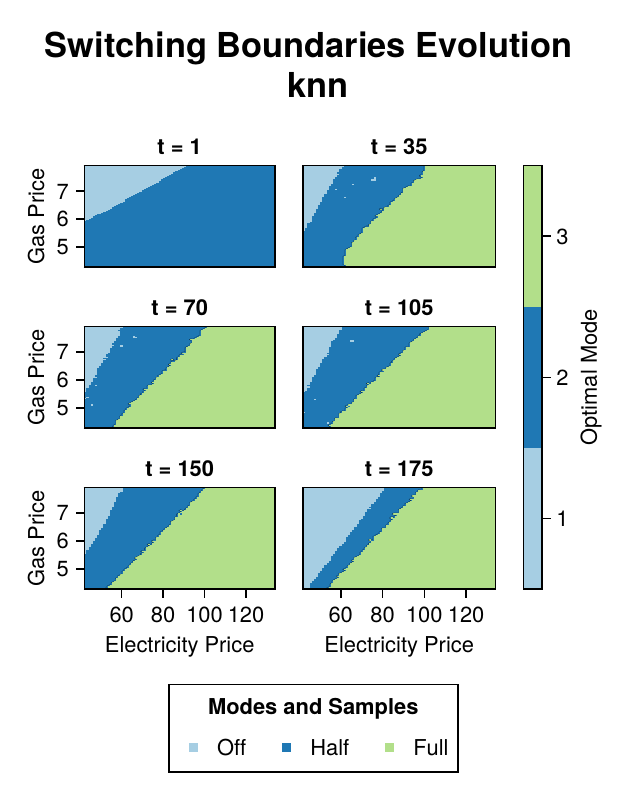}
    \caption{$k$-NN Switching Boundaries}
    \label{fig:carmona_KNN}
  \end{subfigure}
  \hfill
  \begin{subfigure}[b]{0.45\textwidth}
    \centering
    \includegraphics[width=\textwidth]{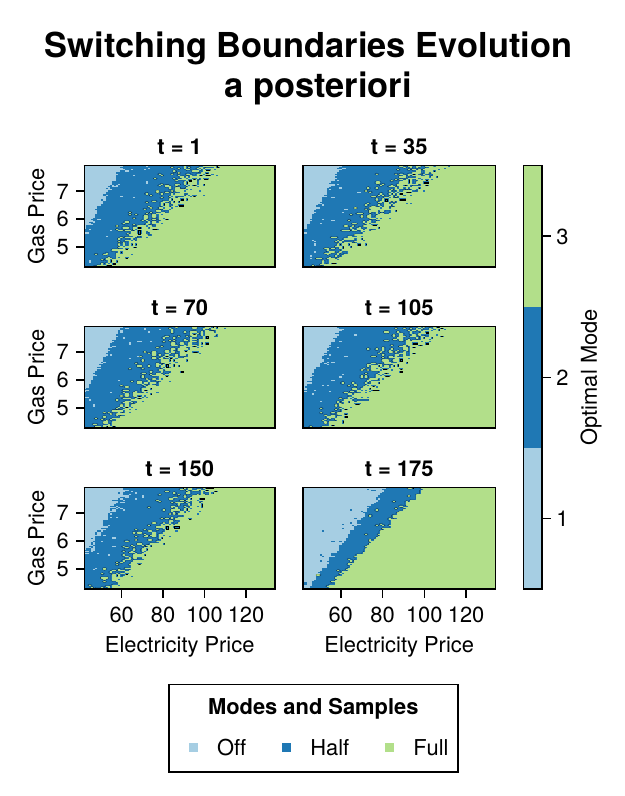}
    \caption{A Posteriori Switching Boundaries}
    \label{fig:carmona_aposteriori}
  \end{subfigure}
  \vskip\baselineskip
  \begin{subfigure}[b]{0.45\textwidth}
    \centering
    \includegraphics[width=\textwidth]{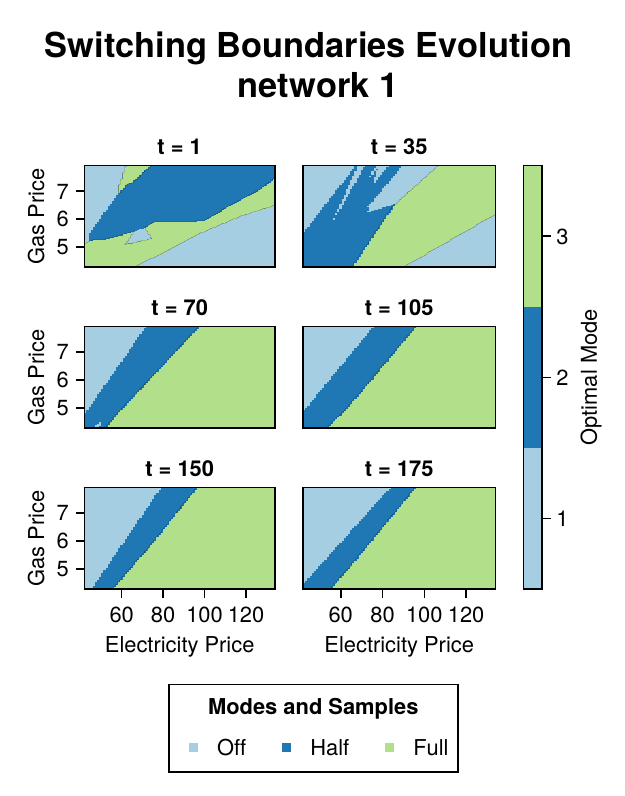}
    \caption{Network Switching Boundaries}
    \label{fig:carmona_network_1}
  \end{subfigure}
  \hfill
  \begin{subfigure}[b]{0.45\textwidth}
    \centering
    \includegraphics[width=\textwidth]{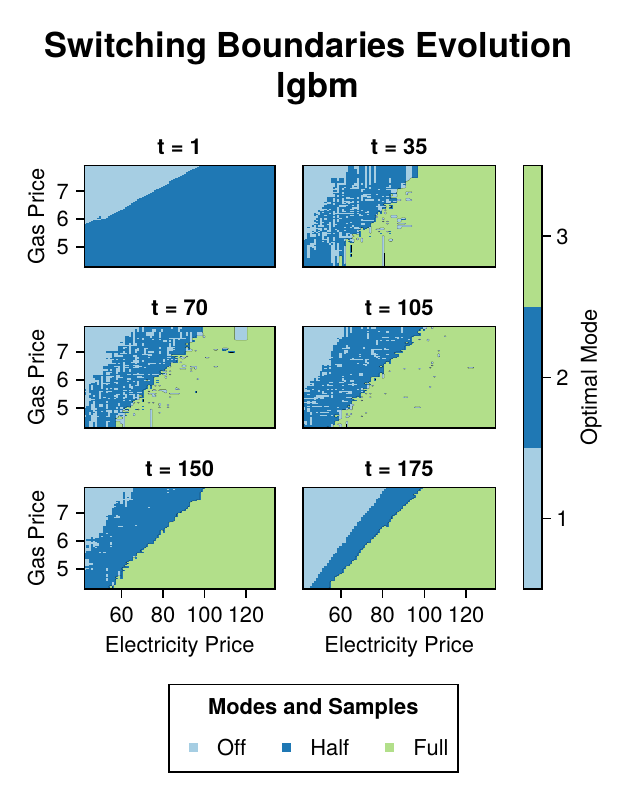}
    \caption{LGBM Switching Boundaries}
    \label{fig:carmona_lgbm}
  \end{subfigure}
  \caption{Evolution of Switching Boundaries Using Different Regression Methods. The plots show the decision boundaries at different time points, when the current mode is 1 (Off) for the CL example \cref{sec:numerical:carmona}. The black part shows in particular the boundary between mode 2 and mode 3.}
  \label{fig:carmona_boundaries}
\end{figure}

\begin{figure}[!htbp]
  \centering
  \begin{tikzpicture}
    \begin{axis}[
        width=0.8\textwidth,
        height=0.4\textwidth,
        xlabel={Time step},
        ylabel={Loss},
        grid=both,
        minor grid style={gray!25},
        major grid style={gray!50},
        legend pos=north east,
        legend style={font=\small},
        title={Training and Validation Loss},
        xmin=1, xmax=180,
        ymin=0, ymax=0.0025,
        y tick label style={/pgf/number format/fixed, /pgf/number format/precision=4}
      ]
      \addplot[blue, thick] table[x=time_step, y=avg_train_loss, col sep=comma] {loss_curve.csv};
      \addlegendentry{Training Loss}
      \addplot[red, thick, dashed] table[x=time_step, y=avg_val_loss, col sep=comma] {loss_curve.csv};
      \addlegendentry{Validation Loss}
    \end{axis}
  \end{tikzpicture}
  \caption{Training and validation loss curves for HCL ($d=40$) of the model PCA-$k$-NN. We see the average training and validation loss over modes $j=1,2,3$, from time-step 1 to 180. We can see that both losses at the start of training, rightmost of figure, begin by rising. Then around time-step 140, losses start to decline.}
  \label{fig:loss_curves}
\end{figure}
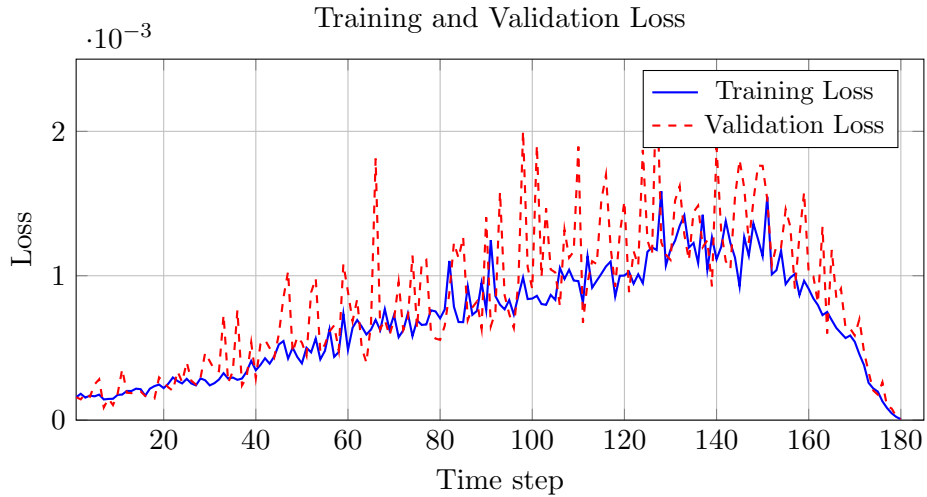

\section{Discussion and conclusion}\label{sec:discussion}
\subsection{Interpretation of numerical results}
The regression models that performed well identified near-optimal decision boundaries and mode-switching behavior. However, for $k$-NN in higher dimensions this boundary detection success came with a systematic prediction bias problem. One explanation is that $k$-NN's local averaging creates constant extrapolation in undersampled regions, failing to capture trends as states approach underexplored regions of state-space.

Decision boundary stability varied between methods. Despite testing several different feedforward neural network architectures (different depths, widths, activations, and regularization schemes), the networks exhibited instability both with regard to hyperparameter changes and across time steps, with decision boundaries shifting significantly between time-steps (\cref{fig:carmona_boundaries}). Most other models demonstrated stable training behavior, producing similar results under small hyperparameter variations, an important advantage when training hundreds of models in the backward induction process.

The loss curves (\cref{fig:loss_curves}) give one potential reason why no single method dominates: regression difficulty changes systematically over time. Because terminal conditions are simple (just equal to 0 in our examples), the first regression problems are easy to solve. At early time-steps in our random processes, the state domain is concentrated, which also makes regression easier. This domain then generally expands due to diffusion, specifically the term $p^c_r$ (in \cref{thm:1,thm:2,thm:3}) increases with time for a fixed $r$. Specifically, in the Longstaff-Schwartz backward induction, each regression inherits the complexity of previous approximations. Methods with high variance or training instability can create erratic continuation value estimates, which may explain why $k$-NN and its stable local averaging maintains better decision boundaries through the backward induction.
\subsection{Theoretical guarantees for \textit{k}-NN Regression}
Our theoretical analysis provides context for the strong empirical performance of $k$-NN. The concentration bounds in \cref{thm:1,thm:2,thm:3} establish that $k$-NN predictions concentrate around the target $\widetilde V_i$ at each regression step. 

\subsection{Limitations}\label{sec:limitations}
Our evaluation used minimal hyperparameter tuning to assess out-of-the-box performance of classical regression models. While systematic optimization might alter relative rankings, this reflects realistic constraints for practitioners who cannot afford extensive tuning.

Our high-dimensional experiments relied on PCA for dimensionality reduction. This choice was not compared against alternatives such as kernel PCA, partial least squares, random projections, or autoencoders, limiting conclusions about the optimal preprocessing approach. However, the strong performance of PCA-$k$-NN suggests that dimensionality reduction can make $k$-NN competitive even when raw $k$-NN performs poorly in high dimension. We did not systematically test PCA with other regression methods, because the other methods either performed well without it, or our goal was to test the inherent dimensionality handling of the model, such as implicit feature selection in tree-based methods.

Further, our neural network experiments were limited to feedforward architectures. Other network structures may be better suited to this class of problems.

\subsection{Implications}\label{sec:implications}
Our results demonstrate that simple, classical methods often suffice for OS problems. In our experiments, $k$-NN, linear models, and tree-based approaches collectively achieved near-optimal performance with greater robustness and fewer training failures than feedforward neural networks. This suggests that practitioners would gain by using classical regression methods before exploring deep learning approaches, which potentially offer minimal performance gains at significant computational and tuning costs.

\subsection{Future work}
Our theoretical analysis opens several promising directions for future research. First, extending our single-step concentration bounds to characterize error propagation through the full backward induction, bounding $|\widehat{V}_i - V_i|$, is a natural next step. Second, our concentration bound requires sub-Gaussian or sub-exponential processes, which limits applicability to jump processes with polynomial tails. Extending the bound to these cases would broaden the scope of the method. Third, while our main theorem requires large $m_Y$ for concentration, our experiments suggest $m_Y = 1$ is sufficient in practice, indicating that the bound likely could be tightened.

The prediction discrepancy we observed with $k$-NN suggests exploring adaptive approaches. Since regression difficulty varies over time steps (see \cref{fig:loss_curves}), hyperparameters could be dynamically adjusted, or different models could be used entirely at different time steps. Our preliminary tests with $k$-NN/linear model mixtures showed promise: good decision boundaries from $k$-NN together with better extrapolation from linear components.

Our approach treats each (time, mode)-pair as an independent regression problem, with only warm-start initialization from previous modes providing continuity in network training. Neural networks could exploit structural relationships more extensively. Specifically, making time an explicit network input, using recurrent architectures designed for sequential data, or jointly learning value functions and their gradients as in \cite{hure2020deep,bayraktar2023neural} could improve stability and performance. In addition, directly estimating decision boundaries rather than value functions, as in \cite{reppen2025neural} for optimal stopping, could also reduce the number of regressions, and potentially improve performance.

\section*{Acknowledgments}
The first author was supported by the Wallenberg AI, Autonomous
Systems and Software Program (WASP) funded by the Knut and Alice Wallenberg Foundation.
The second author was partially supported by the Swedish Research Council grant dnr: 2019--04098.

We acknowledge the use of AI tools to edit and polish the final text in the preparation of this manuscript.
\printbibliography

\clearpage

\appendix

\section{Additional tables}\label{sec:tables}
This section contains tables for the plots in \cref{fig:combined_results}.

\begin{table}[htbp]
  \centering
  \caption{Performance Metrics of Different Strategies, CL Experiment}
  \label{tab:strategy_metrics_3}
  \begin{tabular}{lccc}
    \toprule
    \textbf{Strategy} & \textbf{Final Value Capture} & \textbf{Decision Quality} & \textbf{Internal Consistency} \\
    \midrule
    $k$-NN            & \textbf{0.9949}              & \textbf{0.9455}           & 0.9275                        \\
    LGBM              & 0.9937                       & 0.9455                    & 0.9539                        \\
    Network 1         & 0.9852                       & 0.9219                    & \textbf{0.9716}               \\
    Linear            & 0.9084                       & 0.7453                    & 0.9601                        \\
    Forest            & 0.9056                       & 0.7549                    & 0.9230                        \\
    Greedy            & 0.1960                       & 0.0790                    & 1.0000                        \\
    \bottomrule
  \end{tabular}
  \caption*{\small Note: Bold values indicate the best performance in each metric category.}
\end{table}

\begin{table}[htbp]
  \centering
  \caption{Performance Metrics of Different Strategies, ACLP Experiment}
  \label{tab:strategy_metrics}
  \begin{tabular}{lccc}
    \toprule
    \textbf{Strategy} & \textbf{Final Value Capture} & \textbf{Decision Quality} & \textbf{Internal Consistency} \\
    \midrule
    PCA $k$-NN        & \textbf{0.9956}              & \textbf{0.6831}           & 0.9814                        \\
    Ridge             & 0.9943                       & 0.6733                    & \textbf{0.9875}               \\
    Linear            & 0.9934                       & 0.6482                    & 0.9858                        \\
    LASSO             & 0.9934                       & 0.6443                    & 0.9802                        \\
    Greedy            & 0.9930                       & 0.6607                    & 1.0000                        \\
    Network 1         & 0.9874                       & 0.5254                    & 0.9042                        \\
    Forest            & 0.9796                       & 0.5603                    & 0.9693                        \\
    LGBM              & 0.9476                       & 0.6332                    & 0.9180                        \\
    \bottomrule
  \end{tabular}
  \caption*{\small Note: Bold values indicate the best performance in each metric category.}
\end{table}

\begin{table}[htbp]
  \centering
  \caption{Performance Metrics of Different Strategies, BSP Experiment}
  \label{tab:strategy_metrics_2}
  \begin{tabular}{lccc}
    \toprule
    \textbf{Strategy} & \textbf{Final Value Capture} & \textbf{Decision Quality} & \textbf{Internal Consistency} \\
    \midrule
    LGBM              & \textbf{0.8796}              & 0.8447                    & \textbf{0.9976}               \\
    $k$-NN            & 0.8721                       & \textbf{0.8494}           & 0.9793                        \\
    Network 1         & 0.8641                       & 0.8047                    & 0.8190                        \\
    Network 2         & 0.8461                       & 0.7844                    & 0.8779                        \\
    Forest            & 0.8308                       & 0.7783                    & 0.7895                        \\
    Linear            & 0.5468                       & 0.4514                    & 0.7977                        \\
    Greedy            & 0.0010                       & 0.0000                    & 1.0000                        \\
    \bottomrule
  \end{tabular}
  \caption*{\small Note: Bold values indicate the best performance in each metric category.}
\end{table}

\begin{table}[htbp]
  \centering
  \caption{Performance Metrics of Different Strategies, HCL Dimension = 50}
  \label{tab:strategy_metrics_dim50}
  \begin{tabular}{lccc}
    \toprule
    \textbf{Strategy} & \textbf{Final Value Capture} & \textbf{Decision Quality} & \textbf{Internal Consistency} \\
    \midrule
    $k$-NN            & \textbf{0.9952}              & \textbf{0.9551}           & 0.7782                        \\
    Network 2         & 0.9828                       & 0.9386                    & 0.7006                        \\
    Linear            & 0.9800                       & 0.9031                    & 0.9668                        \\
    Ridge             & 0.9791                       & 0.9013                    & 0.9797                        \\
    Forest            & 0.9732                       & 0.8953                    & \textbf{0.9815}               \\
    LASSO             & 0.9568                       & 0.8632                    & 0.9745                        \\
    LGBM              & 0.8818                       & 0.8609                    & 0.8114                        \\
    Greedy            & 0.7976                       & 0.4289                    & 1.0000                        \\
    \bottomrule
  \end{tabular}
  \caption*{\small Note: Bold values indicate the best performance in each metric category.}
\end{table}

\clearpage

\section{Implementation Details}\label{app:implementation}
For a link to the code for our experiments, see~\cite{code}. The repository includes core implementations of the simulation environments and scripts to run each experiment in \cref{sec:numerical}. While full reproduction requires some manual configuration, we provide documentation to guide the process and code to generate all figures.

\subsection{Software and Packages}
All models except neural networks use \texttt{MLJ.jl}~\cite{Blaom2020} in Julia~\cite{bezanson2017julia}. Specific packages:
\begin{itemize}
  \item Random Forests: \texttt{DecisionTrees.jl}~\cite{ben_sadeghi_2022_7359268}
  \item LightGBM: \texttt{LightGBM.jl} wrapper around~\cite{lightgbm_github}
  \item k-NN: \texttt{NearestNeighbor.jl}
  \item Neural Networks: \texttt{SimpleChains.jl}~\cite{simplechains2023} through \texttt{Lux.jl}~\cite{pal2023lux} interface (optimized for small CPU models, eliminating GPU transfer overhead for $N\times D$ model training)
\end{itemize}

\subsection{Model Configurations}
\cref{table:models} gives an overview of hyperparameters used in simulations.
\begin{table}[ht]
  \centering
  \caption{Machine Learning Models and Parameters}
  \label{table:models}
  \begin{tabular}{ll}
    \toprule
    Model                 & Parameters                                   \\
    \midrule
    $k$-NN$^*$            & $k=10$                                       \\
    PCA-$k$-NN$^*$        & PCA dimensions = 6, $k=10$                   \\
    Random Forest         & Trees = 25, Max depth = 3                    \\
    LightGBM              & Iterations = 200, Learning rate = 0.05,      \\
                          & Leaves = 31, Min data in leaf = 100,         \\
                          & $\lambda_{L2}$ = 0.5, Bagging fraction = 0.8 \\
    Ridge Regression$^*$  & $\lambda = 0.1$                              \\
    Lasso Regression$^*$  & $\lambda = 0.1$                              \\
    Linear Regression$^*$ & Order = $\begin{cases}
                                         6 & \text{if } d=2   \\
                                         1 & \text{otherwise}
                                       \end{cases}$                \\
    Neural Networks       & See \cref{app:networks}                      \\
    \bottomrule
    \multicolumn{2}{l}{\small $^*$Model is precomposed with standardization of input features}
  \end{tabular}
\end{table}

\subsection{Neural Network Architecture Exploration and Final Configurations}\label{app:networks}
\

\textbf{Exploration process.} We tested the following variations:
\begin{itemize}
  \item \textbf{Activations:} ReLU, Tanh
  \item \textbf{Architecture:} 1-4 hidden layers, widths up to 256 units, including pyramid shapes
  \item \textbf{Regularization:} With/without dropout (rates 0.1-0.3 when used)
  \item \textbf{Training:} Various optimizer parameters, with/without early stopping
  \item \textbf{Initialization:} Including the backward initialization scheme described in \cref{sec:models}
\end{itemize}
The reported architectures achieved the best and most stable performance across these variations.
\textbf{Final architectures.}

We use architectures that vary capacity and depth:
\begin{table}[h]
  \centering
  \caption{Neural Network Architectures}
  \begin{tabular}{lll}
    \toprule
    \textbf{Architecture} & \textbf{Layers}                         & \textbf{Usage}               \\
    \midrule
    Low Shallow           & Input $\to$ Dense(32) $\to$ Output              & Low-dimensional experiments  \\
    Low Deep              & Input $\to$ Dense(32) $\to$ Dense(16) $\to$ Output  & Low-dimensional experiments  \\
    High Shallow          & Input $\to$ Dense(128) $\to$ Output             & High-dimensional experiments \\
    High Deep             & Input $\to$ Dense(128) $\to$ Dense(64) $\to$ Output & High-dimensional experiments \\
    \bottomrule
  \end{tabular}
\end{table}

All architectures use:
\begin{itemize}
  \item ReLU activation for hidden layers
  \item Linear activation for output layer
  \item Dropout rate 0.1 after each hidden layer
  \item ADAM optimizer with exponential learning rate decay
\end{itemize}

\textbf{Architecture selection:} Low-capacity networks for experiments in \cref{sec:numerical:carmona,sec:numerical:stapel}. High-capacity networks for higher-dimensional problems in \cref{sec:numerical:carmona_dim,sec:numerical:aid}.

\subsection{Computational Environment}
All simulations were conducted on a system with the following specifications:
\begin{itemize}
  \item \textbf{Processor:} AMD Ryzen Threadripper 1920X (12 cores, 24 threads, 3.5 GHz max frequency)
  \item \textbf{Memory:} 64 GB RAM
  \item \textbf{Operating System:} Ubuntu 20.04.6 LTS
  \item \textbf{Kernel:} Linux 5.15.0-105-generic
  \item \textbf{Architecture:} x86\_64
\end{itemize}

\section{Proof of \cref{thm:1}} \label{app:1}

We will break down the proof of \cref{thm:1} into several lemmas, each one dealing with a different piece of the per-mode error estimate of $\widehat{g}_j(t_n,x)$ with respect to $g_j(t_n,x)$. The value-level bound $|\widehat{V}_i - \widetilde{V}_i| \leq \max_j|\widehat{g}_j - g_j|$ then follows from the max-Lipschitz property, and a union bound over $j$ introduces a factor of $D$. We start with the growth and Lipschitz regularity of $g_j(t_n,x)$.

\begin{lemma} \label{lemma:gj:lipschitz}
  Consider the function $g_j(t_{n},x)$, as defined in \cref{eq:gj}, and assume the conditions as in \cref{thm:1}.
  Then there exists a constant $C_0(d,L,\lambda_1,\lambda_2)$ such that
  \begin{equation*}
    |g_j(t_{n},x)| + |\nabla_x g_j(t_{n},x)| \leq C_0(\|x\|+1).
  \end{equation*}
\end{lemma}
\begin{proof}
  We can write
  \begin{align*}
    \nabla_x g_j(t_{n},x) & = \int \nabla_x \rho(x,t_n;y,t_{n+1}) \widehat{V}_j(t_{n+1},y) \diff y.
  \end{align*}
  From the growth bounds on $\widehat{V}_j$ together with \cref{eq:transition:bounds} and a change of variables we get that there exist constants $C$ (not necessarily the same at every line) and $\lambda > 0$ such that
  \begin{align*}
    \left | \nabla_x g_j(t_{n},x) \right |
     & \leq
    C \int \|\nabla_x \rho(x,t_n;y,t_{n+1})\| (\|y\|+1) \diff y
    \\
     & \leq
    C \int (\|z+x\|+1) \exp(-\lambda\|z\|^2/\Delta t) \diff z
    \leq
    C (\|x\|+1).
  \end{align*}
  A similar calculation shows that $|g_j(t_{n},x)| \leq C(\|x\|+1)$.
\end{proof}

Before we proceed let us state a well known lemma about sub-Gaussian random variables.
\begin{definition} \label{def:sub-gaussian}
  A random variable $X$ is sub-Gaussian with variance proxy $\sigma^2 > 0$ if for all $s \in \R$ we have
  \begin{equation}
    \E[e^{s (X-\E[X])}] \leq e^{\frac{s^2 \sigma^2}{2}}.
  \end{equation}
\end{definition}
\begin{lemma}\label{lem:concentration}
  Let $X_1,\ldots,X_n$ be independent sub-Gaussian random variables with variance proxy $\sigma^2 > 0$. Then for any $\delta > 0$ we have
  \begin{equation*}
    \P \left ( \left | \overline{X}_n - \E[\overline{X}_n] \right | \geq \delta \right ) \leq 2 e^{-\frac{n\delta^2}{2\sigma^2}},
  \end{equation*}
  where $\overline{X}_n = \frac{1}{n} \sum_{i=1}^n X_i$.
\end{lemma}

In the next proof, it is useful to treat the earlier estimate $\widehat V(t_{n+1},\cdot)$ as fixed by conditioning on all randomness present in \cref{alg:main} from time $t_{n+1}$ onward. We define the sigma algebra $\mathcal G_{n+1}$, which is the $\sigma$-algebra generated by $\mathcal X$ and all one-step transitions used at backward induction steps $m \geq n+1$. Conditioning on $\mathcal{G}_{n+1}$, the targets $\widehat V_j(t_{n+1}, \widetilde X^{s,\hat s}_{t_{n+1}})$ have conditional mean $g_j(t_n,X^s_{t_n})$ and are mutually independent across $(s,\hat s)$.
Then, from  \cref{eq:gjtilde,eq:gj}, we can define the conditional expected estimator as
\begin{equation}\label{eq:gjbar}
  \widebar{g}_j(t_n,x) = \E \left [\widetilde{g}_j(t_n,x) \Mid \mathcal{G}_{n+1} \right ] = \frac{1}{k} \sum_{s \in \mathcal{N}_{t_n}(x)} g_j(t_n,X_{t_n}^{s}).
\end{equation}

The following lemma controls the stochastic error of the one-step forward estimate, and while the proof conditions on $\mathcal{G}_{n+1}$, our assumptions gives us a uniform and unconditional bound.

\begin{lemma} \label{lemma:gjtilde:gjbar}
  Consider the functions $\widetilde{g}_j(t_n,x)$ and $\widebar{g}_j(t_n,x)$ as defined in \cref{eq:gjtilde,eq:gjbar}, and assume the conditions as in \cref{thm:1}.
  Then there exists $\sigma(L, \lambda_1) > 0$ such that for any $\delta > 0$ we have
  \begin{equation*}
    \P \left (\sup_{x} |\widetilde{g}_j(t_n,x) - \widebar{g}_j(t_n,x)| \geq \delta \right ) \leq
    2{\binom{M}{k}} \exp \left (-\frac{k m_Y\delta^2}{2 \sigma^2} \right ).
  \end{equation*}
\end{lemma}
\begin{proof}
  Let us begin by conditioning on $\mathcal{G}_{n+1}$, and define $E_1(x) = \widetilde{g}_j(t_n,x) - \widebar{g}_j(t_n,x)$. Then we can write
  \begin{equation*}
    E_1(x) = \frac{1}{k} \sum_{s \in \mathcal{N}_{t_n}(x)} \left (\widetilde{g}_j(t_n,x) - g_j(t_n,X_{t_n}^{s}) \right ) =: \frac{1}{k m_Y} \sum_{s \in \mathcal{N}_{t_n}(x)} \sum_{\hat s=1}^{m_Y} \varepsilon_{(s,\hat s)},
  \end{equation*}
  where each term $\varepsilon_{(s,\hat s)}$ is conditionally independent and conditionally sub-Gaussian with variance proxy $\sigma^2(\lambda_1, L) > 0$. This follows from the growth assumptions on $\widehat{V}_j(t_{n+1},x)$ together with \cref{eq:transition:bounds}.

  Now define $l_i(x) = \frac{\mathbf{1}_{\mathcal{N}_{t_n}(x)}(i)}{k}$ and define the vectors $\mathbf{l}(x) = (l_1(x),\allowbreak\ldots,\allowbreak l_M(x))$ and $\bm{\varepsilon} = (\frac{1}{m_Y}\sum_{\hat s=1}^{m_Y}\varepsilon_{(1,\hat s)},\allowbreak\ldots,\allowbreak\frac{1}{m_Y}\sum_{\hat s=1}^{m_Y}\varepsilon_{(M,\hat s)})$, then we have
  \begin{equation*}
    \sup_x |E_1(x)| = \sup_x |\mathbf{l}(x) \cdot \bm{\varepsilon}| \leq \sup_{\bm{\theta} \in P} \bm{\theta} \cdot \bm{\varepsilon},
  \end{equation*}
  where $P$ is the convex polytope $\{ \bm{\theta} \in \R^M ; \sum_{i=1}^M |\theta_i| \leq 1 , \max_i |\theta_i| \leq 1/k\}$. We note that at each corner $\bm{\theta}$ of $P$ (the set of which we call $P_V$) we have that $\bm{\theta} \cdot \bm{\varepsilon}$ is sub-Gaussian with variance proxy ($\sigma^2/(m_Y k)$), since by independence of $\varepsilon_{(i,j)}$ we have for any $\lambda \in \R$
  \begin{equation} \label{eq:subgaussian}
    \begin{split}
      \E[e^{\lambda \bm{\theta} \cdot \bm{\varepsilon}}]
       & =
      \prod_{i=1}^M \E[e^{\lambda \theta_i \frac{1}{m_Y} \sum_{j=1}^{m_Y} \varepsilon_{(i,j)}}]
      \leq
      \prod_{i=1}^M \prod_{j=1}^{m_Y} e^{\frac{\lambda^2 \theta_i^2 \sigma^2}{2 m_Y^2}} \leq e^{\frac{\lambda^2 (\sum_{i=1}^M \theta_i^2) \sigma^2}{2 m_Y}}
      \\
       & \leq e^{\frac{\lambda^2 \sigma^2}{2 k m_Y}}.
    \end{split}
  \end{equation}
  Now, a linear function inside a convex polytope achieves its maximum at one of the vertices of the polytope, as such we can use the union bound, \cref{eq:subgaussian} and \cref{lem:concentration} to get
  \begin{equation*}
    \P(\sup_x |E_1(x)| \geq \delta) \leq \sum_{\bm{\theta} \in P_V} \P(\bm{\theta} \cdot \bm{\varepsilon} \geq \delta) \leq
    2{\binom{M}{k}} \exp \left (-\frac{k m_Y \delta^2}{2 \sigma^2} \right ).
  \end{equation*}
  Since this is independent of $\mathcal{G}_{n+1}$, we get the unconditional statement by the tower property.
\end{proof}
Under some additional assumptions, like assuming that $\log(\varrho(x,t;y,t+\Delta t))$ is concave w.r.t.~$y$ we can obtain an alternative bound. That is, if we consider the function $\sup_{\bm{\theta} \in P} \bm{\theta} \cdot \bm{\varepsilon}$, we first note that it is a Lipschitz function w.r.t.~$\bm{\varepsilon}$ with Lipschitz constant $1/k$, i.e.
\begin{equation*}
  |\sup_{\bm{\theta} \in P} \bm{\theta} \cdot \bm{\varepsilon}_1 - \sup_{\bm{\theta} \in P} \bm{\theta} \cdot \bm{\varepsilon}_2|
  \leq
  \sup_{\theta \in P} \|\bm{\theta}\| \|\bm{\varepsilon}_1 - \bm{\varepsilon}_2\|
\end{equation*}
where $\sup_{\theta \in P} \|\bm{\theta}\| = 1/\sqrt{k}$ and is attained at one of the vertices of $P$.
From this we get (see~\cite{ledoux2006concentration})
\begin{align*}
  \P(\sup |E_1(x)| \geq \delta)
  \leq &
  \P(\sup_{\bm{\theta} \in P} \bm{\theta} \cdot \bm{\varepsilon} \geq \delta)
  =
  \P \left (\sup_{\bm{\theta} \in P} \bm{\theta} \cdot \bm{\varepsilon} - \E[\sup_{\bm{\theta} \in P} \bm{\theta} \cdot \bm{\varepsilon}] \geq \delta - \E[\sup_{\bm{\theta} \in P} \bm{\theta} \cdot \bm{\varepsilon}] \right )
  \\
  \leq &
  \exp \left ( -\frac{(\delta - \E[\sup_{\bm{\theta} \in P} \bm{\theta} \cdot \bm{\varepsilon}])^2 k m_Y}{\sigma^2} \right ).
\end{align*}
In the above, the term $\E[\sup_{\bm{\theta} \in P} \bm{\theta} \cdot \bm{\varepsilon}]$ is a (Gaussian like) complexity measure of the convex polytope $P$ with respect to $\bm{\varepsilon}$.
To estimate the complexity, we use Hölder's inequality and \cref{eq:transition:bounds} to get that there exists a constant $C$ such that
\begin{equation*}
  \E[\sup_{\bm{\theta} \in P} \bm{\theta} \cdot \bm{\varepsilon}] \leq \E[\sup_{\bm{\theta} \in P} \|\theta\| \| \bm{\varepsilon} \|] \leq \frac{1}{\sqrt{k}} \left ( \E[\|\bm{\varepsilon}\|^2] \right )^{1/2} \leq \frac{\sqrt{M}}{\sqrt{k m_Y}} C.
\end{equation*}
Putting this together yields the estimate whenever $\delta > \frac{\sqrt{M}}{2\sqrt{k m_Y}} C$, for some constant $C_1$,
\begin{equation*}
  \P \left (\sup_{x} |\widetilde{g}_j(t_n,x) - \widebar{g}_j(t_n,x)| \geq \delta \right ) \leq \exp \left (-\frac{k m_Y \delta^2}{C_1} \right ).
\end{equation*}

In order to control the error between the average $k$-nearest estimate and the true continuation value we need the following local random error estimate. One effect is that under our assumptions in \cref{thm:1}, the error is controlled entirely by the information in $\mathcal{X}$ at time $t_n$, as such it ``localizes'' the error in time.

\begin{lemma}\label{lemma:lipthitz:bound}
  Let the functions $g_j(t_n,x)$ and $\widebar{g}_j(t_n,x)$ be as defined in \cref{eq:gj,eq:gjbar}, and assume the conditions as in \cref{thm:1}.
  Let $Q \subset \R^d$ be a bounded set, then there exists a constant $C(d, L, Q,\lambda_1,\lambda_2)$ such that for any $x \in Q$ we have
  \begin{align*}
    |\widebar{g}_j(t_n,x) - g_j(t_n,x)|
    \leq
    C \frac{1}{k} \sum_{s \in \mathcal{N}_{t_n}(x)} \left (\|X_{t_n}^{s}\| + 1 \right ) \|X_{t_n}^{s} - x\|
  \end{align*}
\end{lemma}
\begin{proof}
  From \cref{lemma:gj:lipschitz} and the definition of $\widebar{g}_j(t_n,x)$ we have
  \begin{align*}
    |\widebar{g}_j(t_n,x) - g_j(t_n,x)| \leq
    \frac{1}{k} \sum_{s \in \mathcal{N}_{t_n}(x)} |g_j(t_n,X_{t_n}^{s}) - g_j(t_n,x)| \\
    \leq
    C \frac{1}{k} \sum_{s \in \mathcal{N}_{t_n}(x)} \left (\|X_{t_n}^{s}\| + 1 \right ) \|X_{t_n}^{s} - x\|.
  \end{align*}
\end{proof}

Finally, we are ready to estimate the stochastic error between $g_j$ and $\widebar{g}_j$, this depends on the structure of the set $\mathcal{X}$ at time $t_n$, specifically on how likely it is to see a hole at time $t_n$ in $Q$.

\begin{lemma} \label{lemma:gjbar:gj}
  Consider the functions $g_j(t_n,x)$ and $\widebar{g}_j(t_n,x)$ as defined in \cref{eq:gjbar,eq:gj}, and assume the conditions as in \cref{thm:1}.
  Then, there exist constants $C_1(d, L,Q,\lambda_1,\lambda_2)$ and $C_2(d, L,Q,\lambda_1,\lambda_2)$ such that
  \begin{equation*}
    \P(\sup_{x \in Q} |\widebar{g}_j(t_n,x) - g_j(t_n,x)| \geq \delta)
    \leq
    \begin{cases}
      C_1 \frac{|Q|}{\delta^d} \exp \left (-\frac{M \left (\frac{M-k}{M} - \sup_Q p_{\delta/(4 C_2)}^c \right )_+^2}{2} \right )            & \text{if } \delta/C_2 \leq 1,
      \\
      C_1 \frac{|Q|}{\delta^{d/2}} \exp \left (-\frac{M \left (\frac{M-k}{M} - \sup_Q p_{\sqrt{\delta/(4 C_2)}}^c \right )_+^2}{2} \right ) & \text{if } \delta/C_2 > 1,
    \end{cases}
  \end{equation*}
  where $p_{r}^c(x) = \P(X_{t_n}^{1} \not \in B_{r}(x))$.
\end{lemma}
\begin{proof}
  We use \cref{lemma:lipthitz:bound} to get the existence of a constant $C_2$ such that
  \begin{align*}
    \P(\sup_{x \in Q} |\widebar{g}_j(t_n,x) - g_j(t_n,x)| \geq \delta)
    \leq &
    \P\left (\sup_{x \in Q} \frac{1}{k} \sum_{s \in \mathcal{N}_{t_n}(x)} \left (\|X_{t_n}^{s}\| + 1 \right ) \|X_{t_n}^{s} - x\| \geq \frac{\delta}{C}\right )
    \\
    \leq &
    \P\left (\sup_{x \in Q} (U_n(x)^2+U_n(x)) \geq \frac{\delta}{C_2}\right )
  \end{align*}
  where $U_n(x) = \max_{s \in \mathcal{N}_{t_n}(x)} \|X_{t_n}^{s} - x\|$.
  Now, depending on if $\delta/C_2 \leq 1$ the quadratic term is the dominating term or the linear term, i.e.~we have
  \begin{equation} \label{eq:Un:bound}
    \P\left (\sup_{x \in Q} (U_n(x)^2+U_n(x)) \geq \frac{\delta}{C_2}\right )
    \leq
    \begin{cases}
      \P\left (\sup_{x \in Q} U_n(x) \geq \frac{\delta}{2 C_2}\right )        & \text{if } \delta/C_2 \leq 1,
      \\
      \P\left (\sup_{x \in Q} U_n(x) \geq \sqrt{\frac{\delta}{2 C_2}}\right ) & \text{if } \delta/C_2 > 1.
    \end{cases}
  \end{equation}

  Now, by letting for a set $A$, $\nu_M(A) = \frac{1}{M} \sum_{s=1}^M \mathbf{1}_A(X_{t_n}^{s})$ be the empirical measure of the set $A$, and let $\nu(A) = \P(X_{t_n}^{1} \in A)$ be the true measure. Then for any $\hat \delta > 0$, we have for $\mathcal{A}_{\hat \delta}:=\{B^c_{\hat \delta}(x); x \in Q\}$ the following
  \begin{align*}
    \P \big (\sup_{x \in Q} U_n(x) \geq \hat \delta \big )
    =
    \P \left (\sup_{A \in \mathcal{A}_{\hat \delta}} \nu_M(A) \geq \frac{M-k}{M} \right ).
  \end{align*}
  Now, there exists a set $\mathcal{B}_{\hat \delta} = \{B^c_{\hat \delta/2}(x): x \in P_0\}$ for some finite set of points $P_0$ such that for any $A \in \mathcal{A}_{\hat \delta}$ there is a point $x_0 \in P_0$ such that $A \subset B^c_{\hat \delta/2}(x_0)$. Furthermore, the set $P_0$ contains at most $O(|Q|/\hat \delta^d)$ points. We can now write by the monotonicity of the empirical measure
  \begin{align*}
    \P\left (\sup_{A \in \mathcal{A}_{\hat \delta}} \nu_M(A) \geq \frac{M-k}{M} \right )
    \leq
    \P\left (\max_{B \in \mathcal{B}_{\hat \delta}} \nu_M(B) \geq \frac{M-k}{M} \right ).
  \end{align*}
  Finally, from the triangle inequality, the union bound and Hoeffding's inequality we get
  \begin{equation} \label{eq:hole}
    \begin{split}
      \P \left (\max_{B \in \mathcal{B}_{\hat \delta}} (\nu_M(B) - \nu(B)) \geq \frac{M-k}{M}-\max_{B \in \mathcal{B}_{\hat \delta}} \nu(B) \right )
      \\
      \leq
      \sum_{B \in \mathcal{B}_{\hat \delta}} \P \left (\nu_M(B) - \nu(B) \geq \frac{M-k}{M}-\max_{B \in \mathcal{B}_{\hat \delta}} \nu(B) \right )
      \\
      \leq
      C \frac{|Q|}{\hat \delta^d} \exp \left ( \frac{-M \left ( \frac{M-k}{M} - \max_{B \in \mathcal{B}_{\hat \delta}} \nu(B) \right )_+^2}{2} \right ).
    \end{split}
  \end{equation}

  Now, from \cref{eq:Un:bound,eq:hole} we get in the case $\delta/C_2 \leq 1$ that
  \begin{equation*}
    \P(\sup_{x \in Q} |\widebar{g}_j(t_n,x) - g_j(t_n,x)| \geq \delta)
    \leq C \frac{|Q|}{\delta^d} \exp \left (-\frac{M \left (\frac{M-k}{M} - \sup_Q p_{\delta/(4 C_2)}^c \right )_+^2}{2} \right ),
  \end{equation*}
  and, in the case $\delta/C_2 > 1$ we get
  \begin{equation*}
    \P(\sup_{x \in Q} |\widebar{g}_j(t_n,x) - g_j(t_n,x)| \geq \delta)
    \\
    \leq C \frac{|Q|}{\delta^{d/2}} \exp \left (-\frac{M \left (\frac{M-k}{M} - \sup_Q p_{\sqrt{\delta/(4 C_2)}}^c \right )_+^2}{2} \right ).
  \end{equation*}
\end{proof}

\begin{remark}
  We note that the random variable $U_n(x)$ in the proof above, is the $k$-th order statistic of the random variables $\|X_{t_n}^{s} - x\|$, and in the case the transition density is regularly varying, concentration estimates for the $k$-th order statistic can be found in~\cite{pere2023hill}. This happens for instance in the case of jump-diffusion processes, where the random variable $Y_i$ is fat tailed as in the case of the CL example in \cref{sec:numerical:carmona}.
\end{remark}

Finally, we are ready to tackle the truncated continuation value $\widehat{g}_j(t_n,x)$ defined in \cref{eq:truncation}. This is the final piece of the puzzle in the proof of \cref{thm:1}, and it says that the truncation does not introduce a large error. The reason is that $|\widetilde{g}_j|$ inherits with high probability the growth of $|g_j|$, which is half the truncation level.
\begin{lemma} \label{lemma:gjhat:gjtilde}
  Consider the functions $\widetilde{g}_j(t_n,x)$ and $\widehat{g}_j(t_n,x)$ as defined in \cref{eq:truncation,eq:gjtilde}, and assume the conditions as in \cref{thm:1}. Let the constant $C_0$ in \cref{eq:truncation} be the same as in \cref{lemma:gj:lipschitz}.
  There exists constants $C_1,\, C_2,\, \sigma > 0$, depending only on $d$, $L$, $Q$, $\lambda_1$, and $\lambda_2$,
  such that	for any fixed bounded set $Q \subset \R^d$ and $\delta > 0$ we have that
  \begin{multline*}
    \P(\sup_{x \in Q}  |\widehat{g}_j(t_n,x) - \widetilde{g}_j(t_n,x)| \geq \delta)
    \leq
    2{\binom{M}{k}} \exp \left (-\frac{k m_Y(C_0+\delta)^2}{2 \sigma^2} \right )
    \\
    +
    \begin{cases}
      C_1 \frac{|Q|}{(C_0+\delta)^d} \exp \left (-\frac{M \left (\frac{M-k}{M} - \sup_Q p_{(C_0+\delta)/(4 C_2)}^c \right )_+^2}{2} \right )            & \text{if } (C_0+\delta)/C_2 \leq 1,
      \\
      C_1 \frac{|Q|}{(C_0+\delta)^{d/2}} \exp \left (-\frac{M \left (\frac{M-k}{M} - \sup_Q p_{\sqrt{(C_0+\delta)/(4 C_2)}}^c \right )_+^2}{2} \right ) & \text{if } (C_0+\delta)/C_2 > 1.
    \end{cases}
  \end{multline*}
\end{lemma}
\begin{proof}
  By the symmetric truncation in \cref{eq:truncation},
  \begin{equation*}
    |\widehat{g}_j(t_n,x) - \widetilde{g}_j(t_n,x)| = \bigl(|\widetilde{g}_j(t_n,x)| - 2 C_0(\|x\|+1)\bigr)_+,
  \end{equation*}
  so the event $\{|\widehat{g}_j - \widetilde{g}_j| \geq \delta\}$ equals $\{|\widetilde{g}_j| \geq 2 C_0(\|x\|+1) + \delta\}$. Adding and subtracting $\widebar{g}_j(t_n,x)$ and $g_j(t_n,x)$, and using $|g_j(t_n,x)| \leq C_0(\|x\|+1)$ from \cref{lemma:gj:lipschitz},
  \begin{equation*}
    |\widetilde{g}_j(t_n,x)| \leq |\widetilde{g}_j - \widebar{g}_j| + |\widebar{g}_j - g_j| + C_0(\|x\|+1).
  \end{equation*}
  Hence $|\widetilde{g}_j(t_n,x)| \geq 2C_0(\|x\|+1) + \delta$ implies $|\widetilde{g}_j - \widebar{g}_j| + |\widebar{g}_j - g_j| \geq C_0(\|x\|+1) + \delta \geq C_0 + \delta$, so by a union bound
  \begin{align*}
    \P\!\left(\sup_{x \in Q} |\widehat{g}_j(t_n,x) - \widetilde{g}_j(t_n,x)| \geq \delta\right)
    & \leq \P\!\left(\sup_{x \in Q} |\widetilde{g}_j(t_n,x) - \widebar{g}_j(t_n,x)| \geq (C_0+\delta)/2\right)
    \\
    & \quad + \P\!\left(\sup_{x \in Q} |\widebar{g}_j(t_n,x) - g_j(t_n,x)| \geq (C_0+\delta)/2\right).
  \end{align*}
  The proof is complete after applying \cref{lemma:gjtilde:gjbar,lemma:gjbar:gj}.
\end{proof}

\begin{proof}[Proof of \cref{thm:1}]
  Since $\widehat{V}_i = \max_j[a_j + \widehat{g}_j]$ and $\widetilde{V}_i = \max_j[a_j + g_j]$ with $a_j = \Delta t\, f_j(t_n,x) - c_{ij}(t_n,x)$, the $1$-Lipschitz property of max gives
  \begin{equation*}
    |\widehat{V}_i(t_n,x) - \widetilde{V}_i(t_n,x)| \leq \max_{j=1,\dots,D} |\widehat{g}_j(t_n,x) - g_j(t_n,x)|.
  \end{equation*}
  It therefore suffices to bound $|\widehat{g}_j - g_j|$ for each fixed $j$. By the triangle inequality,
  \begin{align*}
    |\widehat{g}_j(t_n,x) - g_j(t_n,x)|
    \leq &
    |\widetilde{g}_j(t_n,x) - \widebar{g}_j(t_n,x)|
    + |\widebar{g}_j(t_n,x) - g_j(t_n,x)|
    \\
         & + |\widehat{g}_j(t_n,x) - \widetilde{g}_j(t_n,x)|
    \\
    =:   &
    E_1(x) + E_2(x) + E_3(x).
  \end{align*}
  As such by the union bound
  \begin{align*}
    \P \left (\sup_{x \in Q} |\widehat{g}_j(t_n,x) - g_j(t_n,x)| \geq \delta \right )
    \leq &
    \P \left (\sup_{x \in Q} E_1(x) \geq \delta/3 \right )
    +
    \P \left (\sup_{x \in Q} E_2(x) \geq \delta/3 \right )
    \\
         & +
    \P \left (\sup_{x \in Q} E_3(x) \geq \delta/3 \right ).
  \end{align*}
  By \cref{lemma:gjtilde:gjbar,lemma:gjbar:gj,lemma:gjhat:gjtilde}, each term is bounded for a fixed $j$. Taking a union bound over $j = 1,\dots,D$ introduces a factor of $D$ and completes the proof.
\end{proof}

\section{Proof of \cref{thm:2}} \label{app:2}

The proof of \cref{thm:2} is simpler than that of \cref{thm:1}, we will modify \cref{lemma:gjtilde:gjbar,lemma:gjbar:gj} to account for the fact that $x$ is fixed which drastically improves the concentration of the error.

\begin{lemma} \label{lemma:gjtilde:gjbar:2}
  Consider the functions $\widetilde{g}_j(t_n,x)$ and $\widebar{g}_j(t_n,x)$ as defined in \cref{eq:gjtilde,eq:gjbar}, and assume the conditions as in \cref{thm:1}.
  Then there exists $\sigma(L, \lambda_1) > 0$ such that for any $\delta > 0$ we have
  \begin{equation*}
    \P \left (|\widetilde{g}_j(t_n,x) - \widebar{g}_j(t_n,x)| \geq \delta \right ) \leq 2 \exp \left (-\frac{k m_Y \delta^2}{2 \sigma^2} \right ).
  \end{equation*}
\end{lemma}
\begin{proof}
  As in the proof of \cref{lemma:gjtilde:gjbar} we have
  \begin{equation*}
    \P \left (|\widetilde{g}_j(t_n,x) - \widebar{g}_j(t_n,x)| \geq \delta \Mid \mathcal{G}_{n+1} \right ) =
    \P \left ( |E_1(x)| \geq \delta \Mid \mathcal{G}_{n+1} \right )
  \end{equation*}
  where
  \begin{equation*}
    E_1(x) = \frac{1}{k m_Y} \sum_{s \in \mathcal{N}_{t_n}(x)} \sum_{\hat s=1}^{m_Y} \varepsilon_{(s,\hat s)}.
  \end{equation*}
  Recall that each term $\varepsilon_{(s,\hat s)}$ is conditionally mean zero independent, and is sub-Gaussian with variance proxy $\sigma^2(\lambda_1, L) > 0$. Now by \cref{lem:concentration} we get that
  \begin{equation*}
    \P \left ( |E_1(x)| \geq \delta \Mid \mathcal{G}_{n+1} \right ) \leq 2 \exp \left (-\frac{k m_Y \delta^2}{2 \sigma^2} \right ),
  \end{equation*}
  and the tower property completes the proof.
\end{proof}

\begin{lemma} \label{lemma:gjbar:gj:2}
  Consider the functions $\widebar{g}_j(t_n,x)$ and $g_j(t_n,x)$ as defined in \cref{eq:gj,eq:gjbar}, and assume the conditions as in \cref{thm:1}.
  Then, there exist constants $C_1(d, L,|x|,\lambda_1,\lambda_2)$ and $C_2(d, L,|x|,\lambda_1,\lambda_2)$ such that
  \begin{equation*}
    \P(|\widebar{g}_j(t_n,x) - g_j(t_n,x)| \geq \delta)
    \leq
    \begin{cases}
      2 \exp \left (-\frac{M \left (\frac{M-k}{M} - p_{\delta/(4 C_2)}^c \right )_+^2}{2} \right )        & \text{if } \delta/C_2 \leq 1,
      \\
      2 \exp \left (-\frac{M \left (\frac{M-k}{M} - p_{\sqrt{\delta/(4 C_2)}}^c \right )_+^2}{2} \right ) & \text{if } \delta/C_2 > 1.
    \end{cases}
  \end{equation*}
  where $p_{r}^c(x) = \P(X_{t_n}^{1} \not \in B_{r}(x))$.
\end{lemma}
\begin{proof}
  As in the proof of \cref{lemma:gjbar:gj} we need to estimate
  $\P(U_n(x) \geq \hat \delta)$ where $U_n(x) = \max_{s \in \mathcal{N}_{t_n}(x)} \|X_{t_n}^{s} - x\|$. Proceeding similarly but without the covering argument we get from Hoeffding's inequality that
  \begin{equation*}
    \P(U_n(x) \geq \hat \delta) \leq 2 \exp \left (-\frac{M \left (\frac{M-k}{M} - p_{\hat \delta}^c \right )_+^2}{2} \right ).
  \end{equation*}
\end{proof}

Similarly, \cref{lemma:gjhat:gjtilde} can be modified to account for the fixed $x$ using the above lemmas. As in the proof of \cref{thm:1}, the max-Lipschitz bound $|\widehat{V}_i - \widetilde{V}_i| \leq \max_j|\widehat{g}_j - g_j|$ reduces the problem to per-mode estimates, and a union bound over $j = 1,\dots,D$ completes the proof of \cref{thm:2}.

\section{Proof of \cref{thm:3}}
\begin{definition} \label{def:subexponential}
  A random variable $X$ is sub-exponential with parameter $\lambda > 0$ if
  \begin{equation}
    \E[e^{s (X-\E[X])}] \leq e^{\frac{s^2 \lambda^2}{2}},
  \end{equation}
  for all $|s| \leq 1/\lambda$.
\end{definition}
We begin by stating the following classical variant of \cref{lem:concentration} from probability theory, which is a direct consequence of the definition of sub-exponential random variables and the Chernoff bounding method:
\begin{lemma} \label{lem:subexponential}
  Let $X_1,\ldots,X_n$ be independent sub-exponential random variables with parameter $\lambda > 0$, then for any $\delta > 0$ we have for some constant $C > 0$ depending on $\lambda$ that
  \begin{equation*}
    \P \left ( \left \lvert \overline{X}_n-\E[\overline{X}_n] \right \rvert  \geq \delta \right ) \leq
    2\max \left \{ e^{-\frac{n \delta^2}{2 \lambda^2}}, e^{-\frac{n \delta}{2 \lambda}} \right \},
  \end{equation*}
  where $\overline{X}_n = \frac{1}{n} \sum_{i=1}^n X_i$.
\end{lemma}

Proceeding in a similar way as in the proof of \cref{thm:1}, we first note that \cref{lemma:gj:lipschitz} still holds as the upper bounds in \cref{eq:transition:bounds:subexp} are of the form $K(\|x-y\|)$, and as such the change of variables argument still holds. In the proof of \cref{lemma:gjtilde:gjbar} we note that, conditional on $\mathcal{G}_{n+1}$, the random variables $\varepsilon_{(s,\hat s)}$ are now sub-exponential with a parameter depending only on $\lambda_1$ and $L$, and as such we can use \cref{lem:subexponential} instead of \cref{lem:concentration}. Since the resulting conditional bound is uniform in $\mathcal{G}_{n+1}$, the tower property yields
\begin{equation*}
  \P\left (\sup_x |\widetilde{g}_j(t_n,x) - \widebar{g}_j(t_n,x)| \geq \delta \right ) \leq 2 {\binom{M}{k}} \max \left \{ e^{-\frac{k m_Y \delta^2}{2 \sigma^2} }, e^{-\frac{k m_Y \delta}{2 \sigma}}\right \}.
\end{equation*}
\cref{lemma:lipthitz:bound,lemma:gjhat:gjtilde} follow in the same way as before, and the max-Lipschitz bound with union bound over $j = 1,\dots,D$ completes the proof of \cref{thm:3}.

\section{Jump diffusion processes and verification of (\ref{eq:transition:bounds:subexp})} \label{app:3}

In this section we will verify \cref{eq:transition:bounds:subexp} for the following $1$d-jump diffusion process
\begin{equation} \label{eq:jumpdiffusion}
  X_t = x + \int_0^t b(X_s) \diff s + \int_0^t \sigma(X_s) \diff W_s + \sum_{i=1}^{\infty} Y_i \mathbf{1}_{\{T_i \leq t\}},
\end{equation}
where $b, \sigma : \R \to \R$ are Lipschitz continuous functions, $W_t$ is a standard Brownian motion, $Y_i$ are i.i.d.~random variables which have mean zero and finite moments of all orders with sub-exponential density $\varphi$. The jump times $T_i$ are the arrival times of a Poisson process $\{N_t\}$ with intensity $\lambda > 0$.

Define the convolution density
\begin{equation*}
  q_t(x,y) = \int_{\R} \varphi(y-z) p_t(x,z) \diff z,
\end{equation*}
where $p_t(x,y)$ is the transition density of the diffusion part of the process, specifically $p_t(x,y)$ is the transition density for
\begin{equation*}
  Z_t = x + \int_0^t b(Z_s) \diff s + \int_0^t \sigma(Z_s) \diff W_s.
\end{equation*}

\begin{proposition}[\cite{kohatsu2022density}]
  For any $t > 0$ and any $x,y \in \R$, the transition density $f_t(x,y)$ of $X_{t}$ given $X_0 = x$, is given by the following convolution
  \begin{align*}
    f_t(x,y) = & p_t(x,y)e^{-\lambda t}
    \\
               & +
    \sum_{n=1}^\infty \int_{0 < t_1 < \ldots < t_n < t < t_{n+1}}
    q_{t_1-t_0} \star \ldots \star q_{t_n-t_{n-1}} \star p_{t-t_n} \lambda^{n+1} e^{-\lambda t_{n+1}} \diff t_1 \ldots \diff t_{n+1},
  \end{align*}
  where
  \begin{equation*}
    (g_1 \star g_2)(x,y) = \int_{\R} g_1(x,z) g_2(z,y) \diff z.
  \end{equation*}
\end{proposition}

From the above proposition we note that the derivative of the transition density with respect to $x$ will then be given by
\begin{align*}
  \partial_x f_t(x,y) = & \partial_x p_t(x,y)e^{-\lambda t}
  \\
                        & +
  \sum_{n=1}^\infty \int_{0 < t_1 < \ldots < t_n < t < t_{n+1}}
  \partial_x q_{t_1-t_0} \star \ldots \star q_{t_n-t_{n-1}} \star p_{t-t_n} \lambda^{n+1} e^{-\lambda t_{n+1}} \diff t_1 \ldots \diff t_{n+1},
\end{align*}
where
\begin{equation*}
  \partial_x q_{t_1-t_0} = \int_{\R} \varphi(y-z) \partial_x p_{t_1-t_0}(x,z) \diff z.
\end{equation*}

If we assume the same gradient bounds for $p_t$ as we did earlier in \cref{eq:transition:bounds}, which is a Gaussian upper bound, we can use the same arguments as in the proof of the upper density bound as in~\cite{kohatsu2022density} to get the following result.

\begin{proposition} \label{prop:jumpdiffusion:bounds}
  There exist constants $C, c > 0$ such that for any $t > 0$ and any $x,y \in \R$, the derivative of the transition density $f_t(x,y)$ of $X_{t}$ given $X_0 = x$ satisfies
  \begin{align*}
    |\partial_x f_t(x,y)| \leq \frac{C}{\sqrt{t}} e^{- c|x-y| }.
  \end{align*}
\end{proposition}

As such we have that the transition density satisfies the bounds in \cref{eq:transition:bounds:subexp} for any $x,y \in \R$ and $t > 0$.

\end{document}